\documentclass{amsart}

\usepackage[english]{babel}
\usepackage{lineno} 

\usepackage[colorlinks]{hyperref}
\hypersetup{
    linkcolor=red, 
    citecolor=blue, 
}

\usepackage{multicol}
\usepackage{multirow}
\usepackage{graphicx}

\usepackage{amsmath}

\usepackage[latin1]{inputenc}
\usepackage{amsthm}
\usepackage{amsfonts}
\usepackage{amssymb}
\usepackage{enumitem}
\usepackage{xcolor}

\newtheorem*{theorem*}{Theorem}
\newtheorem{theorem} {Theorem}[section]
\newtheorem{corollary}[theorem]{Corollary}
\newtheorem{lemma}[theorem]{Lemma}
\newtheorem{proposition}[theorem]{Proposition}

\theoremstyle{definition}
\newtheorem{example}[theorem]{Example}
\newtheorem{remark}[theorem]{Remark}

\newtheorem{definition}[theorem]{Definition}

\newcommand{\length}  {\operatorname{length} }
\newcommand{\conv}  {\operatorname{conv} }

\newcommand{\vertices}  {\operatorname{vertices} }
\newcommand{\Vol}  {\operatorname{Vol} }
\newcommand{\Surf}  {\operatorname{Surf} }

\newcommand{\Z}{\mathbb{Z}}
\newcommand{\Q}{\mathbb{Q}}
\newcommand{\R}{\mathbb{R}}
\newcommand{\N}{\mathbb{N}}
\newcommand{\T}{\mathbb{T}}

 \usepackage[colorinlistoftodos,bordercolor=orange,backgroundcolor=orange!20,linecolor=orange,textsize=tiny]{todonotes}

\newcommand{\fracV}{\frac1V }
\newcommand{\fracnone}{}

\begin{document}

\setcounter{page}{1}


\title{The complete classification of empty lattice $4$-simplices}
\author{\'Oscar Iglesias-Vali\~no \and Francisco Santos}

\address
{
Departamento de Matem\'aticas, Estad\'istica y Computaci\'on,
Universidad de Cantabria,
39005 Santander, Spain
}
\email{francisco.santos@unican.es, oscar.iglesias@unican.es}

\thanks{Supported by grants MTM2014-54207-P and MTM2017-83750-P (both authors) and BES-2015-073128 (O.Iglesias) of the Spanish Ministry of Economy and Competitiveness. F.~Santos is also supported by 
the Einstein Foundation Berlin under grant EVF-2015-230}

\begin{abstract}
An empty simplex is a lattice simplex with only its vertices as lattice points. 
Their classification in dimension three was completed by G.~White in 1964. 
In dimension four, the same task was started in 1988 by S.~Mori, D.~R.~Morrison, and I.~Morrison, with their motivation coming from the close relationship between empty simplices and terminal quotient singularities. They conjectured a classification of empty simplices of prime volume, modulo finitely many exceptions. 
Their conjecture was proved by Sankaran (1990) with a simplified proof by Bober (2009). 
The same classification was claimed by Barile et al.~in 2011 for simplices of non-prime volume, but this statement was proved wrong by Blanco et al.~(2016+).

In this article, we complete the classification of $4$-dimensional empty simplices. 
In doing so, we correct and complete the classification claimed by Barile et al., and we also compute all the finitely many exceptions, by first proving an upper bound for their volume. 
The whole classification has:

\begin{enumerate}
\item One $3$-parameter family, consisting of simplices of width equal to one.
\item Two $2$-parameter families (the one in Mori et al., plus a second new one).
\item Forty-six $1$-parameter families (the 29 in Mori et al., plus 17 new ones).
\item 2461 individual simplices not belonging to the above families, with (normalized) volumes ranging between 24 and 419.
\end{enumerate}

We characterize the infinite families of empty simplices in terms of the lower dimensional point configurations that they project to, with techniques that can be applied to higher dimensions and larger classes of lattice polytopes.
\end{abstract}

\maketitle

\setcounter{tocdepth}{1}

\section{Introduction and statement of results}
\label{sec:Introduction}

Let $\Lambda \subset \R^d$ be a linear lattice. 
For most of what we do there is no loss of generality in taking $\Lambda= \Z^d$. 
An \emph{empty $d$-simplex} (often called an \emph{elementary simplex} in the algebraic geometry literature) is a $d$-simplex $P=\conv(v_0,\dots,v_d)$ such that $P\cap \Lambda = \{v_0,\dots,v_d\}$. 
Empty simplices are a particular class of \emph{lattice polytopes}, that is, polytopes with vertices in $\Lambda$. 
In this paper we complete the classification of $4$-dimensional empty simplices, a task started more than thirty years ago by Mori, Morrison and Morrison~\cite{MMM}. 

Classifying is meant modulo affine isomorphism of the lattice. 
That is, two lattice $d$-polytopes $P$ and $P'$ (with respect to lattices $\Lambda$ and $\Lambda'$) are called \emph{isomorphic} if there is an affine map $f:\R^d\to \R^d$ with $f(P)=P'$ that induces a lattice isomorphism $f|_\Lambda: \Lambda \overset{\cong}{\to} \Lambda'$.
For example, \emph{unimodular $d$-simplices}, that is, simplices whose vertices form an affine $\Z$-basis of $\Lambda$, are exactly the lattice polytopes isomorphic to the standard simplex $\Delta_d:=\conv(0, e_1,\dots, e_d)$ in the standard lattice $\Z^d$.

Interest in empty simplices and their classification comes from two sources.

From the perspective of discrete geometry and the geometry of numbers, empty simplices can be considered the ``building blocks'' of  lattice polytopes, since every lattice polytope can be triangulated into empty simplices. 
In particular, understanding the structure of empty simplices in any fixed dimension $d$ allows one to derive consequences for all lattice $d$-polytopes. 
A first example of this in dimension two is Pick's formula relating the area and number of lattice points in a lattice polygon~\cite{BeckRobins}, which can easily be derived from Euler's formula and the fact that every empty triangle is unimodular. 
Examples where the classification of \emph{3-dimensional} empty simplices is applied to derive general results on lattice $3$-polytopes are \cite{KantorSarkaria,SantosZiegler,CodenottiSantos,5points}.

From the perspective of algebraic geometry, empty simplices are (almost) in bijection with the \emph{terminal quotient singularities} that arise in the minimal model program.
Indeed, all abelian quotient singularities are toric and their corresponding cone is simplicial, that is, spanned by a linear basis $v_1,\dots,v_d$ of $\Q^d$. 
Taking (without loss of generality) the $v_i$'s to be the primitive integer vectors along their rays, the singularity is \emph{terminal} if and only if the simplex $\conv(0,v_1,\dots,v_d)$ is empty. 
Thus, isomorphism classes of terminal quotient singularities of dimension $d$ are in bijection with isomorphism classes of ``empty $d$-simplices with a marked vertex'', where the marked vertex is the one to be placed at the origin. 
In particular, the classification of $3$-dimensional empty simplices, sometimes dubbed the \emph{terminal lemma}, was instrumental in the completion of the minimal model program in dimension three by Mori~\cite{Mori85,Mori88}.
See~\cite{MS,MMM,Borisov-class,Borisov, Dais} for more on this relation.

The classification of empty simplices of dimension two is trivial: as said above, all empty triangles are unimodular.
In  dimension three there are infinitely many non-isomorphic empty tetrahedra, but they admit a simple two-parameter classification.
In this classification and in the rest of the paper we use the word volume always meaning \emph{normalized volume}, so that the volume of a unimodular simplex is 1 and the volume of every lattice polytope is an integer.
In the correspondence with quotient singularities, the normalized volume of a lattice simplex equals the multiplicity of the corresponding quotient singularity.

\begin{theorem}[\cite{White}, see also~\cite{MS,Sebo}]
\label{thm:white}
Let $q\in \N$. 
Every empty tetrahedron $T$ of volume $q$ is isomorphic to 
\[
T(p,q):= \conv \{ (0,0,0), (1,0,0), (0,0,1),(p,q,1) \}
\]
for some $p\in\Z$ with $\gcd(p,q)=1$. 
In particular, there is an affine lattice functional taking value $0$ at two vertices of $T$ and value $1$ at the other two.
Moreover, $T(p,q)$ is isomorphic to $T(p',q)$ if and only if $p'=\pm p^{\pm 1} \pmod q$.
\end{theorem}

In dimension four the classification is much more complicated, as we will see. 
In order to state it we need to introduce some concepts and notation. 
For a given lattice $d$-simplex $P=\conv(v_0,\dots,v_d)$, we denote by $\Lambda_P$ the affine lattice generated by $\{v_0,\dots,v_d\}$. 
Without loss of generality we assume $\Lambda_P$ to contain the origin (e.g., we translate $P$ so that $v_0=0$) and thus we have a quotient group $G_P:= \Lambda/\Lambda_P$, which is a finite abelian group of order equal to the volume of $P$. 
We say that $P$ is \emph{cyclic} if $G_P$ is a cyclic group.

\begin{theorem}[Barile et al.~\protect{\cite[Theorem~1]{BBBK11}}]
\label{thm:cyclic}
All empty simplices of dimension four are cyclic.
\end{theorem}

This result suggests the use of a generator of $G_P$ as a simple and compact invariant classifying empty $4$-simplces. 
Let us be more precise.
For each $p\in \R^d$, the \emph{barycentric coordinates} of $p$ with respect to the simplex $P=\conv(v_0,\dots,v_d)$  are the unique $(d+1)$-tuple $b=(b_0,\dots,b_d)\in \R^{d+1}$ such that $\sum b_i=1$ and $\sum b_i v_i=p$. 
Two points represent the same class modulo $\Lambda_P$ if, and only if, their barycentric coordinates differ by an integer vector. 
In particular, we call \emph{barycentric coordinates of a class $[p]\in G_P$} the barycentric coordinates of any representative $p\in [p]$, considered as a vector in $(\R/\Z)^{d+1}$. 
One can choose a canonical representative for barycentric coordinates by requiring all coordinates to be in $[0,1)$, but it will  typically be more convenient for us to choose a representative with sum of coordinates equal to zero. 

Observe that the barycentric coordinates of every $p\in \Lambda$ lie in $\frac1V \Z^{d+1}$, where $V$ is the  volume of $P$. Thus, we can express every $p\in \Lambda$ as an integer vector in $\Z^{d+1}$, namely $V$ times its barycentric coordinates.

\begin{definition}
\label{defi:tuple}
A \emph{$(d+1)$-tuple} for a cyclic $d$-simplex $P$ of volume $V$ is the vector of barycentric coordinates of any generator of $G_P$, multiplied by $V$. 
It is an integer vector in $\Z^{d+1}$ whose coordinates we consider modulo $V$.
\end{definition}

\begin{corollary}
\label{coro:cyclic}
Every cyclic $d$-simplex (in particular, every empty $d$-simplex with  $d\le 4$)  can be characterized by its volume $V$ and a $(d+1)$-tuple in ${\Z_V}^{d+1}$.
Two cyclic simplices $P$ and $P'$ of the same volume $V$ are isomorphic if and only if their tuples can be obtained from one another via mutiplication by a unit in $\Z_V$ and/or permutation of coordinates.
\qed
\end{corollary}

\begin{example}[Empty $3$-simplices]
\label{example:White}
The  $3$-simplex  of Theorem~\ref{thm:white}, $T(p,q)= \conv \{ (0,0,0), (1,0,0), (0,0,1),(p,q,1) \}$, has the associated $4$-tuple $(p,-p,-1,1)$, since $(0,1,0)$ is a generator for $G_P\cong \Z_q$ and
\[
(0,1,0) =  \left(1+\frac{p}q\right) (0,0,0) - \frac{p}q(1,0,0)   - \frac1q(0,0,1) +  \frac1q(p,q,1).
\]
\end{example}

One last ingredient that is useful in order to state (and prove) our classification is to look at what is the smallest dimension of a hollow projection of a given simplex. 
Here, a \emph{hollow} polytope is a lattice polytope with no interior lattice points and a hollow projection of a hollow lattice polytope $P\subset \R^d$ is an affine map $\pi: \R^d \to \R^k$ such that $\pi(P)$ is hollow with respect to the lattice $\pi(\Lambda)$.
It is obvious that the only hollow $1$-polytope is the unimodular simplex (a unit segment) and it is easy to show (see e.~g.~\cite[Theorem 2]{Schicho}) that the only hollow $2$-polytope that does not project to a unit segment is the second dilation of a unimodular triangle. We denote it $2\Delta_2$ and it is displayed in Figure~\ref{fig:2Delta}.
\begin{figure}[htb]
\label{fig:2Delta}
\includegraphics{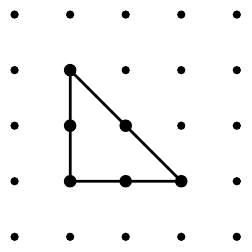}
\caption{The second dilation of a unimodular triangle $\Delta_2$, which is the only hollow $2$-polytope not projecting to a unit segment.}
\end{figure}
We can now state our main theorem:

\begin{theorem}[Classification of empty $4$-simplices]
\label{thm:main}
Let $P$ be an empty $4$-simplex of volume $V\in \N$ and let $k\in \{1,2,3,4\}$ be the minimum dimension of a hollow polytope that $P$ projects to. Then $P$ is as follows, depending on $k$:

\begin{itemize}

\item[\framebox{$k=1$:}] 
$P$ lies in the three-parameter family parametrized by $V$ and another two integer parameters $\alpha,\beta$ with $\gcd(\alpha,\beta,V)=1$; the $5$-tuple of $P$ is  $(\alpha+\beta,-\alpha,-\beta,-1,1)$. 

\item[\framebox{$k=2$:}] 
$P$ lies in one of the following two two-parameter families parametrized by $V$ and another
integer parameter $\alpha$ with $\gcd(\alpha,V)=1$:
\begin{align*}
 (1,-2,\alpha,-2\alpha,1+\alpha) &\quad\text{ with odd $V$,} \quad \text{ and } \\
\frac{V}2\left(0, 1 ,0, 1, 0\right)+ \fracnone (-1,-1,\alpha,-\alpha,2) &\quad\text{ with $V\in 4\Z$}.
\end{align*}
We call the first family \emph{primitive} and the second \emph{nonprimitive}.

\item[\framebox{$k=3$:}] Except for finitely many simplices (of volumes bounded by 72, see~Prop.~\ref{prop:sporadic-k=3}) $P$ belongs to one of the 29 \emph{primitive} + 17 \emph{nonprimitive} families with $5$-tuples shown in Tables~\ref{table:primitive} and~\ref{table:nonprimitive},
parametrized by the volume $V$ alone (plus a choice of sign in some of the nonprimitive families). 

The volume needs to satisfy the modular conditions stated in the caption of Table \ref{table:primitive} and in Table \ref{tab:nonprimitive-bipyr} (from Section~\ref{sec:k=3}), respectively.

\item [\framebox{$k=4$:}] There are finitely many possibilities for $P$, with their volumes bounded by 419.
See more details in Theorem~\ref{thm:sporadic}, below.
\end{itemize}
\end{theorem}

\begin{remark}[From $d+1$-tuples to coordinates]
\label{remark:quintuple}
For the convenience of the reader, here comes an explicit recipe to translate a volume $V$  and a tuple $b=(b_0,\dots,b_d)\in {\Z_V}^{d+1}$ to actual coordinates for a cyclic $d$-simplex $P$ that they represent. 
%
%
Suppose one of the entries in $b$, say the first one $b_0$, is a unit modulo $V$; this property is equivalent to the corresponding facet of $P$ being unimodular (see Lemma~\ref{lemma:facetvolumes}) and it is a fact that all empty $4$-simplices have at least two such unimodular facets (Corollary~\ref{coro:2unimodular}).
Then, since we can multiply $b$ by a unit modulo $V$ there is no loss of generality in assuming $b_0=-1$. Also, since the entries of $b$ are important only modulo $V$ and add up to zero, without loss of generality we assume that $\sum_{i=0}^d b_i=V$. In these conditions, the simplex $P$ can be taken to be
 \[
\conv(e_1,\dots ,e_d, v),
\]
where $v=(b_1,\dots,b_d)$. Indeed, this simplex is clearly of volume $V$ (the last vertex lies at lattice distance $\sum_{i=1}^d b_i -1 = V$ from the facet spanned by the standard basis, which is unimodular) and it is represented by our tuple since the origin has barycentric coordinates $\frac{1}{V} b$ for it:
\[
(0,0,0,0) = \frac{b_1}{V}e_1 + \cdots + \frac{b_d}{V}e_d - \frac{1}{V}v.
\]

For a concrete example, consider the first $5$-tuple $(9,1, -2, -3, -5)$ of Table~\ref{table:primitive} and let $V=100$. We first modify $b$ to have sum of entries equal to $V$ and one entry $-1$ (we do this with the first entry, but it could be done with the second or fourth):
\begin{align*}
11 \cdot (9,1, -2, -3, -5) &= (-1,11,-22,-33,-55) \\ &= (-1, 11, -22, 67, 45) \pmod{100}.
\end{align*}
Then, the simplex $P$ can be taken to be
\[
\conv(e_1,e_2,e_3,e_4, (11,-22,67,45)).
\]

As another example, the simplex of $5$-tuple $(-1,-\alpha-\beta,\alpha,\beta,1)$ and volume $V$  (obtained  from the case $k=1$ of Theorem~\ref{thm:main} by permuting coordinates and reflecting about the origin) is isomorphic to 
\begin{align*}
\conv\{e_1,e_2,e_3,e_4,(V-\alpha-\beta, \alpha,\beta,1)\}
\cong
\conv\{0, e_2,e_3,e_4, (V,\alpha,\beta,1)\}.
\end{align*}
The last isomorphism is via the map $(x_1,x_2,x_3,x_4) \mapsto (\sum_ix_i-1,x_2,x_3,x_4)$.
\end{remark}

\begin{table}[h]
\hrule
\small
\renewcommand{\arraystretch}{1.4}
$\begin{array}{l}
 (9,1,-2,-3,-5) \\ 
 (9,2,-1,-4,-6) \\ 
 (12,3,-4,-5,-6) \\ 
 (12,2,-3,-4,-7) \\ 
 (9,4,-2,-3,-8) \\ 
 (12,1,-2,-3,-8) \\ 
 (12,3,-1,-6,-8) \\ 
 (15,4,-5,-6,-8) \\ 
 (12,2,-1,-4,-9) \\ 
 (10,6,-2,-5,-9) \\ 
\end{array}
\qquad 
\begin{array}{l}
 (15,1,-2,-5,-9) \\ 
 (12,5,-3,-4,-10) \\ 
 (15,2,-3,-4,-10) \\ 
 (6,4,3,-1,-12) \\ 
 (7,5,3,-1,-14) \\ 
 (9,7,1,-3,-14) \\ 
 (15,7,-3,-5,-14) \\ 
 (8,5,3,-1,-15) \\ 
 (10,6,1,-2,-15) \\
 (12,5,2,-4,-15) \\ 
\end{array}
\qquad
\begin{array}{l}
 (9,6,4,-1,-18) \\ 
 (9,6,5,-2,-18) \\ 
 (12,9,1,-4,-18) \\ 
 (10,7,4,-1,-20) \\ 
 (10,8,3,-1,-20) \\ 
 (10,9,4,-3,-20) \\ 
 (12,10,1,-3,-20) \\ 
 (12,8,5,-1,-24) \\ 
 (15,10,6,-1,-30) \\ 
\end{array}$
\bigskip
\caption{The 29 primitive $1$-parameter families of empty $4$-simplices (they coincide with the \emph{stable quintuples} of Mori, Morrison, and Morrison~\cite{MMM}). 
The necessary and sufficient conditions on $V$ for each $5$-tuple to define an empty simplex are that no factor of $V$ divides two (or more) of the entries in the $5$-tuple. 
E.g., 
in the first one $V\not\in3\Z$ since  $3$ divides two entries of $(9,1,-2,-3,-5)$;
in the last one $V\not\in2\Z\cup 3\Z\cup 5\Z$ since each of $2$, $3$ and $5$ divide 
two or three entries of $(15,10,6,-1,-30)$.
}
\label{table:primitive}
\hrule\end{table}

\begin{table}
\hrule
\small
\setlength\tabcolsep{0pt}
\renewcommand{\arraystretch}{1.4}

\[
\begin{array}{cc}
\begin{array}{l}
 \textbf{\quad\qquad\qquad Index 2:}\\
  \frac{V}2( 0,   0,  1, 1,  0 ) \ + \  
\fracnone  (3,   -1,  -6,    2,    2  ) \\
  \frac{V}2(1,  0,   0,   0,  1 ) \ + \  
 \fracnone  ( 4,   -3,    1,   -4,    2  ) \\
  \frac{V}2(1,  0,   0,   0,  1 ) \ + \  
 \fracnone  ( 2,    3,   -1,  -8,    4  ) \\
  \frac{V}2( 0,  1, 1,  0,   0 ) \ + \  
 \fracnone  ( 1,  -6,    2,   6,   -3  ) \\
  \frac{V}2(1,  0,  1,  0,   0 ) \ + \  
 \fracnone  (6,  -8,    4,   -3,    1  ) \\
  \frac{V}2(1,  0,   0,   0,  1 ) \ + \  
 \fracnone  ( 4,    3,   -1,  -12,   6  ) \\
\\
 \textbf{\quad\qquad\qquad Index 4:}\\
 \pm \frac{V}4(2, 1, 1,  0,   0 ) \ + \  
 \fracnone  (3,  -3,    1,   -2,    1  ) \\
\pm  \frac{V}4( 0, 1, 1,  0,  2 ) \ + \  
 \fracnone  ( 1,   3,   -1,  -6,   3  ) \\
%
\end{array}
& 
\begin{array}{l}
 \textbf{\quad\qquad\qquad Index 3:}\\
\pm  \frac{V}3( 0,   0,  2, 1,  0 ) \ + \  
 \fracnone  (-3,    2,    1,    1,   -1  ) \\
\pm  \frac{V}3(1,  0,  2,  0,   0 ) \ + \  
 \fracnone  ( 3,   -3,    1,   -2,    1  ) \\
\pm  \frac{V}3( 0,   0,  1, 2,  0 ) \ + \  
 \fracnone  (-3,    1,    2,    2,   -2  ) \\
\pm  \frac{V}3( 0,   0,  1, 2,  0 ) \ + \  
 \fracnone  (4,   -2,  -4,    1,    1  ) \\
\pm  \frac{V}3(1,  0,  2,  0,   0 ) \ + \  
 \fracnone  ( 3,  -6,    2,    2,   -1  ) \\
\pm  \frac{V}3(1,  0,  2,  0,   0 ) \ + \  
 \fracnone  (4,  -6,    1,    2,   -1  ) \\
\pm  \frac{V}3(1,  0,  2,  0,   0 ) \ + \  
 \fracnone  (4,   -3,    1,  -4,    2  ) \\
\pm  \frac{V}3( 0,   0,  1, 1, 1 ) \ + \  
 \fracnone  ( 1,  -6,    2,   6,   -3  ) \\
 \textbf{\quad\qquad\qquad Index 6:}\\
\pm  \frac{V}6(1,   0,   0,  4, 1 ) \ + \  
 \fracnone  ( 1,  -3,    1,   2,   -1  ) \\
\end{array}
\end{array}
\]
\medskip
\caption{The 17 non-primitive $1$-parameter families of empty $4$-simplices. 
$V$ needs to be a multiple of the index $I\in \{2, 3, 4, 6\}$ and satisfy certain additional restrictions, specified in Table~\ref{tab:nonprimitive-bipyr}.
}
\label{table:nonprimitive}
\hrule\end{table}

Some comments about the statement of Theorem~\ref{thm:main}:

\begin{itemize}
\item The classification is not irredundant. The same empty simplex may belong to several families, since it may project to different lower dimensional configurations. 

\item By Corollary~\ref{coro:cyclic}, the parameters $\alpha$ and $\beta$ are important only modulo $V$; also, multiplying a $5$-tuple by a unit modulo $V$ does not change the simplex.

\end{itemize}

 In all the families we have stated some restrictions on the volume $V\in \N$ or the parameters $\alpha,\beta\in \Z_V$ (e.g., the condition $\gcd(V,\alpha,\beta)=1$ when $k=1$). 
Without these restrictions the $5$-tuples represent \emph{hollow} cyclic $4$-simplices.
That these restrictions are necessary for emptyness is part of Theorem~\ref{thm:main}, and their sufficiency is shown in Propositions~\ref{prop:sufficent1and2}, \ref{prop:primitiverestrictions} and \ref{prop:nonprimitiverestrictions}. That is, we have the following converse of Theorem~\ref{thm:main}:

\begin{theorem}
\label{thm:sufficient}
All the cyclic $4$-simplices described in Theorem~\ref{thm:main} are empty.
\end{theorem}

In order to have a complete classification we need to enumerate the finitely many exceptions of the cases $k=3,4$. For this, in Section \ref{sec:k=4} we first prove an upper bound for their volume (Theorem~\ref{thm:bound}) and then enumerate all empty simplices up to that volume. This yields:

\begin{theorem}
\label{thm:sporadic}
Apart of the infinite families described in Theorem~\ref{thm:main}, there are exactly 2461 sporadic empty $4$-simplices.
Their volumes range from 24 to 419 and the number of them for each volume is as listed in Table~\ref{table:sporadic}.
\end{theorem}

\begin{table}
\hrule
\small
\setlength\tabcolsep{-1.5pt}
\renewcommand{\arraystretch}{1.2}

\centerline{
\begin{tabular}{c|c|c|c|c|c|c|c|c}
$\begin{array}{rr}V& \# \\ \hline
    24  &  1\\
    27  &  1\\
    29  &  3\\
    30  &  2\\
    31  &  2\\
    32  &  3\\
    33  &  4\\
    34  &  5\\
    35  &  3\\
    37  &  6\\
    38  &  8\\
    39  &  9\\
    40  &  1\\
    41  &  14\\
    42  &  5\\
    43  &  20\\
    44  &  8\\
    45  &  6\\
    46  &  7\\
    47  &  30\\
    48  &  5\\
    49  &  17\\
    50  &  8\\
    51  &  16\\
    52  &  6\\
    53  &  38\\
    54  &  11\\
    55  &  20\\
    56  &  3\\
\end{array}$&
$\begin{array}{rr}V& \# \\ \hline
    57  &  16\\
    58  &  13\\
    59  &  51\\
    60  &  4\\
    61  &  38\\
    62  &  26\\
    63  &  17\\
    64  &  9\\
    65  &  27\\
    66  &  3\\
    67  &  41\\
    68  &  13\\
    69  &  26\\
    70  &  4\\
    71  &  50\\
    72  &  3\\
    73  &  44\\
    74  &  18\\
    75  &  22\\
    76  &  14\\
    77  &  19\\
    78  &  3\\
    79  &  55\\
    80  &  7\\
    81  &  18\\
    82  &  13\\
    83  &  60\\
    84  &  7\\
    85  &  27\\
\end{array}$&
$\begin{array}{rr}V& \# \\ \hline
    86  &  11\\
    87  &  24\\
    88  &  5\\
    89  &  55\\
    90  &  6\\
    91  &  18\\
    92  &  9\\
    93  &  17\\
    94  &  12\\
    95  &  35\\
    96  &  3\\
    97  &  46\\
    98  &  9\\
    99  &  13\\
    100  &  8\\
    101  &  41\\
    102  &  3\\
    103  &  51\\
    104  &  8\\
    105  &  7\\
    106  &  8\\
    107  &  54\\
    108  &  5\\
    109  &  44\\
    110  &  5\\
    111  &  13\\
    112  &  2\\
    113  &  40\\
    114  &  4\\
\end{array}$&
$\begin{array}{rr}V& \# \\ \hline
    115  &  21\\
    116  &  11\\
    117  &  10\\
    118  &  9\\
    119  &  22\\
    120  &  3\\
    121  &  18\\
    122  &  9\\
    123  &  17\\
    124  &  8\\
    125  &  25\\
    127  &  24\\
    128  &  9\\
    129  &  17\\
    130  &  2\\
    131  &  29\\
    132  &  5\\
    133  &  14\\
    134  &  8\\
    135  &  6\\
    136  &  6\\
    137  &  28\\
    138  &  2\\
    139  &  37\\
    140  &  5\\
    141  &  6\\
    142  &  9\\
    143  &  13\\
    144  &  1\\
\end{array}$&
$\begin{array}{rr}V& \# \\ \hline
    145  &  14\\
    146  &  5\\
    147  &  10\\
    148  &  7\\
    149  &  26\\
    150  &  2\\
    151  &  19\\
    152  &  6\\
    153  &  9\\
%
    154  &  3\\
    155  &  12\\
    156  &  2\\
    157  &  11\\
    158  &  10\\
    159  &  9\\
    160  &  3\\
    161  &  13\\
    163  &  17\\
    164  &  6\\
    165  &  1\\
    166  &  7\\
    167  &  18\\
    168  &  3\\
    169  &  13\\
    170  &  2\\
    171  &  6\\
    172  &  3\\
    173  &  15\\
    174  &  3\\
\end{array}$&
$\begin{array}{rr}V& \# \\ \hline
    175  &  8\\
    176  &  4\\
    177  &  5\\
    178  &  2\\
    179  &  21\\
    180  &  1\\
    181  &  13\\
    182  &  5\\
    183  &  5\\
    184  &  5\\
    185  &  7\\
    186  &  2\\
    187  &  7\\
    188  &  5\\
    189  &  2\\
    190  &  2\\
    191  &  8\\
    192  &  1\\
    193  &  12\\
    194  &  3\\
    196  &  4\\
    197  &  13\\
    199  &  11\\
    200  &  4\\
    201  &  3\\
    202  &  2\\
    203  &  7\\
    204  &  1\\
    205  &  4\\
\end{array}$&
$\begin{array}{rr}V& \# \\ \hline
    206  &  4\\
    207  &  2\\
    208  &  1\\
    209  &  10\\
    211  &  4\\
    212  &  2\\
    213  &  3\\
    214  &  2\\
    215  &  5\\
    216  &  1\\
    218  &  5\\
    219  &  4\\
    220  &  1\\
    221  &  3\\
    222  &  1\\
    223  &  7\\
    225  &  2\\
    226  &  4\\
    227  &  9\\
    229  &  6\\
    230  &  3\\
    232  &  1\\
    233  &  9\\
    234  &  1\\
    235  &  3\\
    237  &  1\\
    238  &  2\\
    239  &  3\\
    241  &  6\\
\end{array}$&
$\begin{array}{rr}V& \# \\ \hline
    244  &  2\\
    245  &  3\\
    247  &  3\\
    248  &  3\\
    249  &  2\\
    250  &  1\\
    251  &  5\\
    254  &  1\\
    256  &  2\\
    257  &  3\\
    259  &  2\\
    261  &  1\\
    263  &  7\\
    265  &  1\\
    267  &  1\\
    268  &  1\\
    269  &  2\\
    271  &  4\\
    272  &  1\\
    274  &  1\\
    275  &  1\\
    278  &  2\\
    283  &  2\\
    287  &  1\\
    289  &  4\\
    290  &  1\\
    291  &  1\\
    292  &  1\\
    293  &  5\\
\end{array}$&
$\begin{array}{rr}V& \# \\ \hline
    299  &  2\\
    304  &  1\\
    308  &  1\\
    310  &  1\\
    311  &  1\\
    313  &  1\\
    314  &  1\\
    317  &  1\\
    319  &  2\\
    321  &  1\\
    323  &  1\\
    331  &  1\\
    332  &  1\\
    334  &  2\\
    335  &  1\\
    347  &  1\\
    349  &  2\\
    353  &  1\\
    355  &  1\\
    356  &  1\\
    376  &  1\\
    377  &  2\\
    397  &  1\\
    398  &  1\\
    419  &  1\\
    \\ \\ \\ \\
\end{array}$
\end{tabular}
\bigskip}

\caption{Statistics of the 2461 sporadic empty $4$-simplices by volume. The complete list of pairs $[V,  (b_0, b_1, b_2, b_3, b_4)]$ for them, is in the first author's website:~\url{https://personales.unican.es/iglesiasvo/empty4simplices}, and as an ancillary file supplementing this arXiv version.
}
\label{table:sporadic}
\hrule\end{table}

Summing up, the whole classification of empty $4$-simplices consists of 1 three-parameter family (case $k=1$), 2 two-parameter families (case $k=2$), 1+17=46 one-parameter families (part of case $k=3$) and 2461 individual examples.

\medskip

The structure of the proof (and of the paper) is as follows. In Section \ref{sec:hollow} we show a general scheme to classify hollow polytopes in fixed dimension into families, and a more explicit one for the case of cyclic simplices. In Sections~\ref{sec:k<=2} and \ref{sec:k=3} we implement this approach for empty $4$-simplices that project to hollow $3$-polytopes, thus proving the cases $k\le3$ of Theorem~\ref{thm:main}.
The derivation of cases $k=1$ and $2$ is quite easy, but the case $k=3$ requires a close look at the classification of hollow $3$-polytopes, recently developed in \cite{AverkovKrumpelmannWeltge}, together with an analysis of hollow lifts of hollow polytopes using ideas from \cite{BlancoHaaseHoffmanSantos}.

The case $k=4$ of Theorem~\ref{thm:main} and its more explicit version in Theorem~\ref{thm:sporadic} are proved in Section~\ref{sec:k=4}. 
In it, using the theory of successive minima and covering minima for convex bodies, we prove an upper bound of $5184$ for the volume of empty $4$-simplices that do not project to lower dimension (Theorem~\ref{thm:bound}). More precisely, the case where $P$ has width at least three was already studied in \cite{IglesiasSantos}, and here we look at the case of width two.
Remember that the \emph{(lattice) width} of a convex body $K$ with respect to a lattice $\Lambda$ is the minimum 
length of the interval $f(K)$ where $f:\R^d\to \R$ runs over all lattice functionals in $\Lambda^*\setminus \{0\}$. The width of a lattice polytope $P$ is a nonnegative integer that equals zero if and only if $P$ is lower-dimensional and equals one if and only if $P$ projects to a unit segment (the case $k=1$ in Theorem~\ref{thm:main}).
The width of hollow $d$-simplices is known to be bounded by $O(d\log d)$~\cite{Banaszczyk_etal} and empty $d$-simplices of width $(1+\epsilon) d$ for an $\epsilon>0$ and arbitrarily large $d$ are constructed in \cite{CodenottiSantos2}.

Once the volume bound is proven, a brute force enumeration of all empty $4$-simplices up to that bound gives the statements. Details on how we implemented this enumeration appear in \cite{IglesiasSantos}. The only new ingredient is pruning the output to discard all empty $4$-simplices from the infinite families of Theorem~\ref{thm:main}.

In Section~\ref{sec:facets} we summarize what volumes can arise for the facets of empty $4$-simplices in the different cases of Theorem~\ref{thm:main}. In particular, we observe that every empty $4$-simplex has at least two unimodular facets and very few of them have exactly two (Corollary~\ref{coro:2unimodular}). Part of the interest in looking at facets volume comes from the fact that for an empty $4$-simplex knowing its volume $V$ and surface area $S$ (by which we mean the sum of normalized volumes of the five facets) is equivalent to knowing the Ehrhart polynomial and $h^*$-vector. Indeed, as seen in Proposition~\ref{prop:hstar}:
\[
h^*=(h^*_0, h^*_1, h^*_2, h^*_3, h^*_4) = \left(1,0, \frac{V+S}2-3, \frac{V-S}2+2, 0\right).
\]

Some of the techniques and intermediate results in our proof of Theorem~\ref{thm:main} are valid for all \emph{hollow} $4$-simplices, not only empty ones. 
In the final Section~\ref{sec:hollow2} we state a classification result for hollow $4$-simplices (Theorem~\ref{thm:mainhollow}) along the same lines as Theorem~\ref{thm:main} yet much less explicit. One of the complications is that they are no longer cyclic in general (e.g., the hollow simplex $4\Delta_4$ has quotient group ${\Z_4}^4$); also, although the volume bound in Theorem~\ref{thm:bound} holds for all hollow $4$-simplices with $k=4$, the enumeration of hollow $4$-simplices up to that bound is computationally much more difficult than that of empty ones.

\medskip
Let us put our results and techniques in context. As already mentioned, the classification of empty $4$-simplices was started in 1988 by S.~Mori, D.~R.~Morrison, and I.~Morrison~\cite{MMM}. At that time it was not known that all empty $4$-simplices are cyclic, but they were only interested in those of prime volume, for which cyclicity is obvious. 
This allowed them to specify $4$-simplices via their $5$-tuples.
They showed the existence of the 1+1+29 \emph{primitive} families with three, two and one parameter that we describe in Theorem~\ref{thm:main} and conjectured, based on an exhaustive enumeration up to $V=1600$, that all empty $4$-simplices of \emph{prime} volume belonged to them, with finitely many exceptions of volume $V\le419$. This  conjecture, without the volume bound, was proved by Sankaran~\cite{Sankaran} (although the published version of his paper omits some cases of a lengthy case study) and later, with simpler arguments, by Bober \cite{Bober}. 

Another empirical exploration of empty $4$-simplices was done by Haase and Ziegler~\cite{HZ00} who enumerated them, without the restriction to prime volumes, for $V\le 1000$. Their main interest was not in classifying them but in bounding their width. 
They found 178 empty $4$-simplices of width three (with volumes between 49 and 179) and a single one of width four (with volume 101). They conjectured that these were all the empty $4$-simplices of width larger than two. These conjectures were proven in our previous paper~\cite{IglesiasSantos}.

In 2011, Barile, Bernardi, Borisov and Kantor~\cite{BBBK11} proved that all empty $4$-simplices are cyclic (Theorem~\ref{thm:cyclic} here). 
They also claimed to have proved that all but finitely many empty $4$-simpices have width one or two, but their proof implicitly assumed that cyclicity (weaker than primality of volume) was enough for the classification of empty $4$-simplices with $k\le 3$ to contain only the $1+1 +29$ infinite families found by Mori et al., a statement  that was proved wrong by Blanco et al.~\cite{BlancoHaaseHoffmanSantos}. 

\begin{table}
\hrule
\small
\setlength\tabcolsep{-1.5pt}
\renewcommand{\arraystretch}{1.2}

\centerline{
\begin{tabular}{c|c|c}
$\begin{array}{rr}
V & \# \\ \hline
17  &  \bf 9\\
19  &  \bf 13\\ 
23  &  \bf 28\\ 
29  &  \bf 39\\ 
31  &  \bf 30\\ 
37  &  \bf 50\\ 
41  &  \bf 76\\ 
43  &  \bf 110\\ 
47  &  \bf 100\\ 
53  &  \bf 195\\ 
59  &  \bf 260\\ 
61  &  186\\ 
67  &  205\\
71  &  250\\
73  &  220\\
79  &  275\\
83  &  300\\
89  &  275\\
97  &  230\\
101  &  201\\ 
103  &  255\\
\end{array}$&
$\begin{array}{rr}
V & \# \\ \hline
107  &  270\\
109  &  220\\
113  &  200\\
127  &  120\\
131  &  145\\
137  &  140\\
139  &  185\\
149  &  130\\
151  &  95\\
157  &  55\\
163  &  85\\
167  &  90\\
173  &  75\\
179  &  105\\
181  &  65\\
191  &  40\\
193  &  60\\
197  &  65\\
199  &  55\\
211  &  20\\
223  &  35\\
\end{array}$&
$\begin{array}{rr}
V & \# \\ \hline
227  &  45\\
229  &  30\\
233  &  45\\
239  &  15\\
241  &  30\\
251  &  25\\
257  &  15\\
263  &  35\\
269  &  10\\
271  &  20\\
283  &  10\\
293  &  25\\
311  &  5\\
313  &  5\\
317  &  5\\
331  &  5\\
347  &  5\\
349  &  10\\
353  &  5\\
397  &  5\\
419  &  5\\
\end{array}$
\end{tabular}
\qquad\quad
\begin{tabular}{c|c|c}
$\begin{array}{rr}
V & \# \\ \hline
     17  &  \bf 0\\ 
     19  &  \bf 0\\ 
     23  &  \bf 0\\ 
     29  &  \bf 15\\ 
     31  &  \bf 10\\ 
     37  &  \bf 30\\ 
     41  &  \bf 66\\ 
     43  &  \bf 100\\ 
     47  &  \bf 150\\ 
     53  &  \bf 190\\ 
     59  &  \bf 255\\ 
     61  &  186\\ 
     67  &  205\\
     71  &  250\\
     73  &  220\\
     79  &  275\\
     83  &  300\\
     89  &  275\\
     97  &  230\\
    101  &  201\\ 
    103  &  255\\
\end{array}$&
$\begin{array}{rr}
V & \# \\ \hline
    107  &  270\\
    109  &  220\\
    113  &  200\\
    127  &  120\\
    131  &  145\\
    137  &  140\\
    139  &  185\\
    149  &  130\\
    151  &  95\\
    157  &  55\\
    163  &  85\\
    167  &  90\\
    173  &  75\\
    179  &  105\\
    181  &  65\\
    191  &  40\\
    193  &  60\\
    197  &  65\\
    199  &  55\\
    211  &  20\\
    223  &  35\\
\end{array}$&
$\begin{array}{rr}
V & \# \\ \hline
    227  &  45\\
    229  &  30\\
    233  &  45\\
    239  &  15\\
    241  &  30\\
    251  &  25\\
    257  &  15\\
    263  &  35\\
    269  &  10\\
    271  &  20\\
    283  &  10\\
    293  &  25\\
    311  &  5\\
    313  &  5\\
    317  &  5\\
    331  &  5\\
    347  &  5\\
    349  &  10\\
    353  &  5\\
    397  &  5\\
    419  &  5\\
\end{array}$
\ \ 
\end{tabular}
\medskip
}

\caption{ 
Number of sporadic terminal quotient singularities of prime volume. Left: the results from~\cite[Table 1.14]{MMM}. Right: our results; unless a sysmmetric empty simplex of that volume exists, each entry is five times the corresponding entry from  Table~\ref{table:sporadic}. Discrepant entries are in boldface. }
\label{table:tqs}
\hrule
\end{table}

We have compared our computation of sporadic examples with the one by Mori et al., who list the number of them for each prime volume up to $419$; see the left part of Table~\ref{table:tqs}, which is Table 1.14 in~\cite{MMM}. The right part of the same table is our count of them. This is not exactly the same count as in Table~\ref{table:sporadic} since we are here counting terminal quotient singularities rather than simplices: each simplex is counted as many times as orbits of vertices are there in its affine-unimodular symmetry group.

As seen in the table, there are some discrepancies between our results and those from~\cite{MMM}. We approached the authors of~\cite{MMM} about this issue and I.~Morrison (personal communication) told us that they no longer have their full output, so it is not possible to verify their numbers, or to look at what particular simplices produce the discrepancies. 
Observe that, when there is a discrepancy, the value in \cite{MMM} is higher than ours (with a single exception for $V=47$ that might  be a typographic error). Our guess is that their mistake was not in the enumeration part but in the search for redundancies, where $5$-tuples defining isomorphic simplices may look different, specially when $V$ is not big with respect to the other entries in the $5$-tuple. This guess is consistent with the facts that all discrepancies have $V < 60$ and discrepancies are bigger for smaller values of $V$. Most entries in both tables, and most discrepancies between the two tables, are multiples of five since most simplices have no symmetries.

\subsection*{Acknowledgements}
We thank M\'onica Blanco for the computations summarized in Lemma~\ref{lemma:dim3enum}, and Martin Henk for suggesting the use of Schwarz symmetrization to prove Lemma~\ref{lemma:slice}. 

\section{How to classify hollow polytopes}
\label{sec:hollow}

\subsection{A general classification scheme}

A lattice polytope $P\subset \R^d$ (with respect to a lattice $\Lambda\cong \Z^d)$ is \emph{hollow} if $P\cap \Lambda \subset \partial P$ and \emph{empty} if $P\cap \Lambda=\vertices(P)$. 
If there is a lattice projection $\pi: \R^n\to \R^k$ sending a polytope $P$ to a polytope $Q$ and $Q$ is hollow with respect to the projected lattice $\pi(\Lambda)$, then $P$ is automatically hollow; (the same is not true for empty). In this situation we say that $\pi$, or $Q$, is a \emph{hollow projection} of $P$, and that $P$ is a \emph{lift} of $Q$.
The starting point to a general classification of hollow lattice polytopes is the following result of Nill and Ziegler:

\begin{theorem}[Nill-Ziegler~\protect{\cite[Thm.~1.2]{NillZiegler}}]
\label{thm:NillZiegler}
For each dimension $d$ there is only a finite number of hollow $d$-polytopes that do not  project onto a hollow  $(d - 1)$-polytope.
\end{theorem}

To rephrase this statement we introduce the following definition:

\begin{definition}
\label{defi:coarse}
Let $d\in \N$ be fixed and let $Q$ be a $k$-dimensional lattice hollow polytope that does not project to any $(k-1)$-hollow polytope, with $k\le d$. We call  \emph{coarse family of $Q$} the collection of all hollow $d$-polytopes that have $Q$ as a hollow projection.
\end{definition}

\begin{corollary}
\label{coro:coarse}
The hollow $d$-polytopes of any fixed dimension $d$ belong to a finite number of coarse families.
\end{corollary}

\begin{proof}
There is one family for each of the finitely many polytopes of Theorem \ref{thm:NillZiegler}, for each $k=1,\dots,d$.
\end{proof}

\begin{example}
\label{exm:coarse2dim}
A lattice polytope $P$ projects to a hollow $1$-polytope if and only if $P$ has \emph{width one}. That is, if $P$ is contained between two consecutive parallel lattice hyperplanes. It is easy to check that the only hollow $2$-polytope without that property is the second dilation $2\Delta_2$ of a unimodular triangle. Thus, the coarse classification of hollow lattice $2$-polytopes is as follows:
\begin{itemize}
\item The dilated unimodular triangle $2\Delta_2$ is a coarse family with a single element.

\item The lattice polygons of width one form a second family. Each of them is isomorphic to a trapezoid $\{0\}\times[0,a] \cup \{1\}\times [0,b]$ with $a,b\in \Z_{\ge 0}$ and $a+b>0$. (The trapezoid degenerates to a triangle if $a$ or $b$ equal zero).
\end{itemize}
\end{example}

\begin{example}
\label{exm:coarse3dim}
The coarse classification of hollow $3$-polytopes is:
\begin{itemize}
\item The coarse family of width one; each of which can be expressed as a pair of lattice polygons.
\item The coarse family projecting to $2\Delta_2$. As before, these can be written as  the convex hull of six hollow lattice segments $\{p_i\} \times [a_i,b_i]$ where $p_i$, $i=1,\dots, 6$, are the six lattice points in $2\Delta_2$ and $[a_i,b_i]$ is an integer interval. 
\item Each of the finitely many (by Theorem ~\ref{thm:NillZiegler}) hollow $3$-polytopes that do not project to dimension two is a coarse family in itself.
These were enumerated by Averkov et al.~\cite{AverkovWagnerWeismantel,AverkovKrumpelmannWeltge}, who showed that there are 12 maximal ones. See Theorem~\ref{thm:AKW}.

\end{itemize}
\end{example}

Observe that the families just defined may not be disjoint. For example, the Cartesian product of $2\Delta_2$ with a unit segment belongs to the first two families of Example~\ref{exm:coarse3dim}, since it projects both to $2\Delta_2$ and to a unit segment.

We are interested in a finer classification, which takes into account the number of lattice points. A \emph{hollow configuration} is a finite set $S$ of lattice points such that $\conv(S)$ is a hollow polytope.

\begin{definition}
\label{defi:fine}
Let $d\in \N$ be fixed and let $S$ be a configuration of $n$ lattice points (perhaps with repetition) in $\R^k$, with $n>d\ge k$.
Assume that $\conv(S)$ is hollow but it does not project to a hollow $(k-1)$-polytope.
We call  \emph{fine family of $S$} the collection of all hollow $d$-polytopes with $n$ vertices that admit a lattice projection 
sending $\vertices(P)$ to $S$.
\end{definition}

\begin{corollary}
\label{coro:fine}
All hollow $d$-simplices belong to a finite number of fine families. More generally, for each fixed $n$, all hollow $d$-polytopes with $n$ vertices belong to a finite number of fine families.
\end{corollary}

\begin{proof}
There is one for each multisubset of size $n$ of the lattice points in each of the finitely many polytopes of Theorem \ref{thm:NillZiegler}, for $k=1,\dots,d$.
\end{proof}

\begin{example}
\label{exm:fine2dim}
There are three fine families of hollow lattice $2$-polytopes:
\begin{itemize}
\item The dilated unimodular triangle $2\Delta_2$ is still a fine family with a single element. The corresponding $S$ has size three (the three vertices of $2\Delta_2$).

\item The lattice polygons of width one fall into two fine families, one projecting to the set $S_1=\{0,1,1\}$ ($n=3, k=1$) and one projecting to the set $\{0,0,1,1\}$ ($n=4, k=1$). Members of the first family are isomorphic to a triangle $\{(0,0)\} \cup \{1\}\times [0,b]$, with $b\in \Z_{\ge 1}$. Members of the second family are trapezoids $\{0\}\times[0,a] \cup \{1\}\times [0,b]$ with $a,b\in \Z_{\ge 1}$.
\end{itemize}
\end{example}

\begin{example}
\label{exm:fine3dim}
There are infinitely many fine families of hollow $3$-polytopes of width one, since they can have arbitrarily many vertices and all polytopes in the same fine family have the same number of vertices, by definition.
\end{example}

One key difference between coarse and fine families is that in the latter we fix the number $n$ of vertices. In particular, if we take $n=d+1$ we are looking at hollow simplices.
Observe that in Example \ref{exm:fine2dim} each fine family is parametrized by $n-k-1$ parameters. 
In the next section we analyze this phenomenon in more detail in the case of interest to us.

Let us finish this section by pointing out that these finiteness results are similar in spirit to Theorem~2.1 in \cite{Borisov-class}, which Borisov derives from the following more general statement of Lawrence~\cite{Lawrence}: for any open subset $U$ of the torus $\T^d:=\R^d/\Z^d$ the family of subgroups of $\T^d$ not intersecting $U$ has finitely many maximal elements. The relation is as follows: let $U$  be the interior of the standard simplex in $\T^d$. Then, discrete subgroups $G\in \T^d$ not meeting $U$ correspond to hollow $d$-simplices $P\subset \R^d$ via the correspondence $P \leftrightarrow G_P:=\Lambda/\Lambda_P$. If $G$ is not discrete (e.g., $G$ corresponds to  positive dimensional linear subspace $V\le \R^d$) then the discrete subgroups of $G$ form a fine family of hollow simplices, in the sense of Definition~\ref{defi:fine}.

\subsection{The case of cyclic simplices}
\label{sec:cyclic}
In this section we relate the $(d+1)$-tuple of a cyclic simplex $P$ to a hollow projection.
Let us fix the following notation:

Let $P=\conv(v_0,\dots,v_d)\subset \R^d$ be a cyclic lattice $d$-simplex of  volume $V$, and let $\Lambda_P$ be the affine lattice generated by its vertices (we assume without loss of generality that $0\in \Lambda_P$). By definition of cyclic simplex, the quotient group $G(P):= \Lambda/\Lambda_P$ is cyclic of order $V$. Let $\pi: \R^d \to \R^k$ be a linear projection and denote
\[
S:=\pi(\vertices(P)) = \{\pi(v_0), \dots, \pi(v_d)\}.
\]
Observe that both $\vertices(P)$ and $S$ are considered as ordered sets, and their ordering corresponds to the order of coordinates in a $(d+1)$-tuple representing $P$.

Let $\Lambda_S$ be the affine lattice generated by $S$, which is a sublattice of $\pi(\Lambda)$. Then  $\pi(\Lambda) / \Lambda_S$ is a cyclic group too, since $\pi$ induces a surjective homomorphism 
\[
\tilde \pi : \Lambda / \Lambda_P \twoheadrightarrow \pi(\Lambda) / \Lambda_S.
\]
Let $I$ be the index of $\Lambda_S$ in $\pi(\Lambda)$ which, by the above remark, divides $V$.
We say that $S$, and the fine family defined by it, are \emph{primitive} if $I=1$; that is, if $\Lambda_S=\pi(\Lambda)$.

We need the following elementary fact about cyclic groups:

\begin{lemma}
\label{lem:cycliclemma}
Let $\pi: \Z_V \to \Z_I$ be a surjective homomorphism between the cyclic groups of orders $V$ and $I$. Then, for every generator $q$ of $\Z_I$ threre is a generator $p$ of $\Z_V$ with $\pi(p)=q$.
\end{lemma}

\begin{proof}
Take as $p$ any prime not dividing $V$ from the arithmetic progression $\{q+nI: n\in \Z\}$. Such a prime exists since, by Dirichlet's prime number theorem, the arithmetic progression contains infinitely many primes.
\end{proof}

\begin{proposition}
\label{prop:tuple}
With the above notation, let $q\in \pi(\Lambda)$ be a generator of the quotient group $\pi(\Lambda) / \Lambda_S$. Then:
\begin{enumerate}
\item There is a vector $a \in \frac1I \Z^{d+1}$ such that 
\[
q= \sum_{i=0}^d a_i\pi(v_i),
\quad\text{and}\quad
1= \sum_{i=0}^d a_i.
\]

\item There is a generator $p \in \Lambda$ of the quotient group $\Lambda / \Lambda_P$ such that the barycentric coordinates of $p$ with respect to $\{v_0,\dots,v_d\}$ have the form
\[
a + \frac1V  b,
\]
with $b\in \Z^{d+1}$  the coefficient vector of an integer affine dependence on $S$.
\end{enumerate}
\end{proposition}

\begin{proof}
For part (1), observe that since $\Lambda_S$ has index $I$ in $\pi(\Lambda)$, we have $\pi(\Lambda) \le \frac1I \Lambda_S$. In particular, the point $q\in \pi(\Lambda)$ can be written as an affine combination, with coefficients in $\frac1I \Z$,  of the points in $S$. The vector $a$ is the vector of coefficients in this dependence.

For part (2), let $p\in \Lambda$ be a generator of $\Lambda/\Lambda_P$ with $\pi(p) = q$, which exists by Lemma~\ref{lem:cycliclemma}.
Let $c= (c_0,\dots, c_d)\in \frac1V  \Z^{d+1}$ be the barycentric coordinates of $p$ with respect to $\{v_0,\dots,v_d\}$. 
That is, $\sum c_i=1$ and $\sum c_i v_i =p$. By construction, $c-a \in  \frac1V  \Z^{d+1}$. The only thing that remains to be shown is that $b:=V(c-a)\in  \Z^{d+1}$ is the coefficient vector of an affine dependence among the $\pi(v_i)$s. This is easy:
\[
\sum_{i=0}^d \left(c - a\right)_i \pi(v_i) =
\pi \left( \sum_{i=0}^d c_i v_i \right) -
\sum_{i=0}^d  a_i  \pi(v_i) =
\pi(p)-q=0.
\]
and 
\[
\sum_{i=0}^d \left(c - a\right)_i  =
 \sum_{i=0}^d c_i  -
\sum_{i=0}^d  a_i =
1-1=0.
\]
\end{proof}

The above statement implicitly gives a parametrization of the fine family of cyclic simplices projecting to $S$. Let us make it more explicit.

\begin{corollary}
\label{coro:parameters}
Let $\Lambda_0$ be a lattice in $\R^k$ and let $S$ be a multiset of $d+1$ lattice points affinely spanning $\R^k$. 
Assume that $\Lambda/\Lambda_S$ is cyclic, of index $I$, and let $a$ be as in part (1) of Proposition~\ref{prop:tuple}. 
Then, the cyclic $d$-simplices of a given volume $V\in I\cdot \N$ and projecting to $S$ are parametrized as having $(d+1)$-tuples
\[
Va +  b,
\]
where $b\in \Z^{d+1}$ runs over all the integer affine dependences of $S$. Moreover, $b$ is only important modulo $V$, and satisfies $\gcd(V,b_0,\dots,b_d)=1$.
\qed
\end{corollary}

The proof of Theorem~\ref{thm:main} is by applying Corollary~\ref{coro:parameters} to the case of cyclic empty $4$-simplices, that is, by looking at hollow configurations of five points in $\R^k$, $k<4$. 
Observe that for primitive families (that is, for $I=1$), the only generator $q$ of $\Lambda/\Lambda_S$ is the zero class, represented (for example) by the first element of $S$. This gives us $a=(1,0,\dots,0)$ but, since we are interested in the tuples modulo the integers, we can as well take $a=0$. This is our convention in all the primitive families of Theorem~\ref{thm:main}.

Observe that if $\conv(S)$ is hollow then all the cyclic simplices of  Corollary~\ref{coro:parameters}  are automatically \emph{hollow}, but not necessarily \emph{empty}. 
Let us now address the issue of the restrictions needed for them to be empty. 
They are related to the volumes of facets and the following observation.

\begin{lemma}
\label{lemma:facetvolumes}
Let $P$ be a cyclic $d$-simplex of volume $V$ with $(d+1)$-tuple $(b_0,\dots,b_d)$. Then, the volume of the $i$th facet of $P$ ($i=0,\dots,d)$ equals
\medskip\\
\centerline{
\hfill
$V_i:=\gcd(V,b_i).
\hfill \qed$%
}
\end{lemma}

\begin{proposition}
\label{prop:coprime}
Let $P$ be a cyclic $d$-simplex of volume $V$ with tuple $ (b_0,\dots,b_d)$. A necessary condition for $P$ to be empty is that no $d-2$ of the $b_i$s (equivalently, no $d-2$ of the facet volumes $V_i$) have a factor in common with $V$.
\end{proposition}

\begin{proof}
Recall that the $V$ tuples $j b$, $j=0,\dots,V-1$, represent the $V$ classes of lattice points in $\Lambda/\Lambda_P$. If $d-2$ of the $b_i$s have a factor in common with $V$ then there is a $j\ne 0$ such that $jb$ has three (or less) nonzero entries. That implies one of the non-zero classes in $\Lambda/\Lambda_P$ to have representatives in a 2-plane spanned by a 2-face of $P$, which implies $P$ has a 2-face that is not unimodular, hence not empty. That is a contradiction since every face of an empty simplex is empty.
\end{proof}
\begin{proposition}
\label{prop:emptyfacets}
Let $P$ be a cyclic hollow $4$-simplex of volume $V$ with $5$-tuple $(b_0,\dots,b_4)$ and, as above, let $V_i:= \gcd(V, b_i)$ (the volume of the \emph{$i$-th} facet of $P$).
The following are equivalent:
\begin{enumerate}
\item $P$ is empty.

\item For each $i$, if $V_i\ne 1$ then the multiset $\{b_0,\dots,b_4\}$ coincides,  modulo $V_i$,
 with the multiset $\{0,\alpha, -\alpha, \beta,-\beta\}$ for some $\alpha$ and $\beta$ coprime with $V_i$.
\end{enumerate}
\end{proposition}

\begin{proof}
Once we know that $P$ is hollow, it will be empty if and only if its facets are empty tetrahedra. The (classes of) lattice points in $\Lambda/\Lambda_P$ lying in the hyperplane of the $i$-th facet are those that have a zero in the $i$-th position of their barycentric coordinates; these, as multiples of the generator $\fracV (b_0,\dots,b_4)$ for the quotient group, are precisely the multiples of $\frac1{V_i}(b_0,\dots,b_4)$. 
The necessary and sufficient condition for the facet to be empty is, by Example~\ref{example:White}, that the four non-zero entries in  
$\frac1{V_i}(b_0,\dots,b_4)$ come in two pairs of opposite entries modulo $V_i$, and that the entries are prime with $V_i$.
\end{proof}

\section{Proof of the main theorem, cases $k=1,2$}
\label{sec:k<=2}

\begin{proof}[Proof of Theorem~\ref{thm:main}, case $k=1$]
There are two possibilities for a hollow configuration $S$ of five points in dimension one. Either $S=\{0,0,0,1,1\}$ or 
$S=\{0,0,0,0,1\}$. We first show that every cyclic simplex projecting to the latter projects also to the former.
Indeed, let $P=\conv(v_0,\dots,v_4)$ be a cyclic simplex and $\pi: \R^4 \to \R$ a lattice projection sending $v_0,\dots,v_3$ to $0$ and $v_4$ to $1$.  Since the facet $F=\conv(v_0,\dots,v_3)$ is an empty tetrahedron, by Theorem~\ref{thm:white} there is a lattice functional $f:\R^4\to \R$ sending two of its vertices (say $v_0$ and $v_1$) to $0$ and the other two to $1$. Let $c=f(v_4)\in \Z$. Then the functional $f-c\cdot \pi$  sends $v_0$, $v_1$ and $v_4$ to $0$ and $v_2$ and $v_3$ to $1$.

That is to say, we have a single fine family, projecting to $S=\{0,0,0,1,1\}$. It is a primitive configuration and the linear space of affine dependences among its five points equals
\[
\{(\alpha+\beta,-\alpha,-\beta,-\gamma,\gamma): \alpha,\beta,\gamma \in \R\}. 
\]
Thus, by Corollary~\ref{coro:parameters}, every cyclic simplex of volume $V$ projecting to $S$ has a $5$-tuple of the form
\[
 (\alpha+\beta,-\alpha,-\beta,-\gamma,\gamma), 
\]
with $\alpha,\beta,\gamma \in \Z$. Fix such a $5$-tuple.  By Proposition~\ref{prop:coprime}, for the simplex to be empty we need
$\gcd(\gamma,V)=\gcd(\alpha,\beta,V)=1$. Multiplying the $(d+1)$-tuple by the inverse of $\gamma$ modulo $V$, there is no loss of generality in taking $\gamma=1$. 
\end{proof}

\begin{proof}[Proof of Theorem~\ref{thm:main}, case $k=2$]
Our set $S$ consists now of five of the six lattice points in $2\Delta_2$ (see Figure~\ref{fig:2Delta}), perhaps with repetition. 
In order for $S$ not to project to dimension 1, we need to use the three vertices of $\Delta$, which leaves six
possibilities for the two additional elements of $S$, modulo affine symmetry.

{\bf Claim:} \emph{no four of the five elements of $S$ can be on the same edge of $2\Delta_2$:} Suppose that $\pi: P \to 2\Delta$ projects an empty $4$-simplex $P$ to $2\Delta$, with four of the vertices going to the same edge of $2\Delta$. Let  $f:\R^2 \to \R$ be the lattice functional taking the value $0$ on that edge and the value $2$ at the opposite vertex. 
Let $\tilde f:= f\circ \pi: \R^4 \to \R$. Since the facet of $P$ where $\tilde f$ vanishes is an empty tetrahedron, by Theorem~\ref{thm:white} there is a lattice functional $g:\R^4\to \R$ sending two of its vertices to $0$ and the other two to $1$. Let $c$ be the value of $g$ at the fifth vertex of $P$. Then $g - \lfloor\frac{c}2\rfloor \tilde f$  takes values $0$ or $1$ at all vertices of $P$, contradicting the fact that $P$ does not project to a hollow segment.

The claim implies that $S$ is, modulo symmetries of $2\Delta_2$, one of the two point configurations in Figure~\ref{fig:2Delta-cases}. 
\begin{figure}
\label{fig:2Delta-cases}
\includegraphics[scale=0.7]{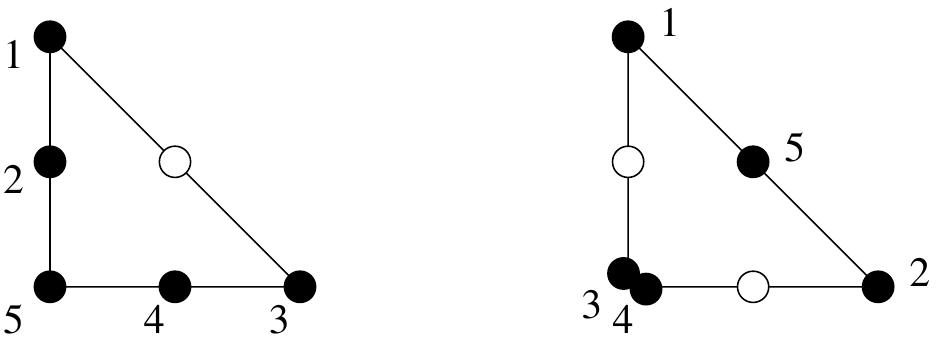}
\bigskip \\
\caption{The six possibilities for a size $5$ subconfiguration of $2\Delta_2$ containing the three vertices. Only the first two arise as the projection of empty $4$-simplices with $k=2$.}
\end{figure}
Their respective spaces of linear dependences are as follows, where coordinates are labeled as shown in the figure.
\[
\{(\beta,-2\beta,\alpha,-2\alpha,\alpha+\beta): \alpha,\beta \in \R^2\}
\qquad{and}\qquad
\{(-\beta,-\beta,\alpha,-\alpha,2\beta): \alpha,\beta \in \R^2\};
\]
 the integer dependences are the same, with $\alpha,\beta\in \Z$. The first configuration is primitive ($I=1$), but in the second one we have $I=2$ and we can choose as barycentric coordinates for the unique generator of the quotient group the vector $\left(0,\frac12,0,\frac12,0\right)$.
Thus, by Corollary~\ref{coro:parameters}, the cyclic simplices of volume $V$ projecting to these configurations are parametrized by 
\begin{align*}
 (\beta,-2\beta,\alpha,-2\alpha,\alpha+\beta)  \quad \text{ and } \quad
\frac{V}2 \left(0,1, 0,1, 0\right)+  (-\beta,-\beta,\alpha,-\alpha,2\beta),
\end{align*}
respectively. In the first case $V$ must be odd, by Proposition~\ref{prop:coprime}. In the second case $V$ must be even, since $V$ is a multiple of the index $I=2$. Proposition~\ref{prop:coprime} also implies that $\gcd(\alpha,V)=\gcd(\beta,V)=1$ for empty simplices. This allows us to multiply the $5$-tuple by the inverse of $\beta$ modulo $V$, producing $5$-tuple in the form of Theorem~\ref{thm:main}. (In the second one, observe that both $\beta$ and its inverse are odd, so that multiplying by $\beta^{-1}$ leaves the part $\left(0,\frac12,0,\frac12,0\right)$ intact). Finally, in the second $5$-tuple $V$ must be a multiple of four since for $V = 2 \pmod4$ the $5$-tuple
\[
\left(0,\frac{V}2,0,\frac{V}2,0\right)+  (-1,-1,\alpha,-\alpha,2)=
  \left(-1,\frac{V}2-1,\alpha,\frac{V}2-\alpha,2 \right)
\]
has two even $b_i$s, contradicting Proposition~\ref{prop:coprime} (observe that $\alpha$ is odd, since $\gcd(\alpha, V)=1$ and $V$ is even). 
\end{proof}

Let us finally check that the conditions stated in Theorem~\ref{thm:main} for $\alpha,\beta$ and $V$ are not only necessary but also sufficient for the corresponding simplices to be empty. To show this via Proposition~\ref{prop:emptyfacets} we need to look at the facets of volumes in each case:

\begin{lemma}
\label{lemma:volumes1and2}
Let $P$ be an empty simplex as in Theorem~\ref{thm:main} with $k\in \{1,2\}$. Then, the volumes $(V_0,V_1,\dots,V_4)$ of its facets are:
\begin{enumerate}
\item If $k=1$, we have $V_0=\gcd(V,\alpha+\beta)$, 
$V_1=\gcd(V,\alpha)$,  $V_2=\gcd(V,\beta)$ and $V_3=V_4=1$. In particular, there can be up to three nonunimodular facets.
\item If $k=2$ then $V_0=V_1=V_2=V_3=1$. In the primitive case $V_4=\gcd(V,\alpha+1)$ and in the nonprimitive case $V_4=2$.
In particular, there is at most one nonunimodular facet.
\end{enumerate}
\end{lemma}

\begin{proof}
This follows directly from Lemma~\ref{lemma:facetvolumes} and the expression of the $5$-tuples, taking into account that if $k=2$ and $P$ is primitive then $V$ is required to be odd, while if $k=2$ and $P$ is nonprimitive then $V$ is required to be a multiple of four.
\end{proof}

\begin{proposition}
\label{prop:sufficent1and2}
All the cyclic simplices in the conditions stated in parts $k=1$ and $k=2$ of Theorem~\ref{thm:main} are empty.
\end{proposition}

\begin{proof}
By Corollary~\ref{coro:parameters}, all integer values of $V$, $\alpha$ and $\beta$ produce hollow simplices, since they produce simplices projecting to hollow configurations in dimensions $1$ and $2$. Hence, we can apply Proposition~\ref{prop:emptyfacets}, taking into account the facet volumes computed in Proposition~\ref{lemma:volumes1and2}:
\begin{itemize}
\item For the $5$-tuple  $ (\alpha+\beta,-\alpha,-\beta,-1,1)$ we have:
\begin{gather*}
(\alpha+\beta,-\alpha,-\beta,-1,1) = (0, -\alpha,\alpha,-1,1) \pmod{\gcd(V,\alpha+\beta)},\\
(\alpha+\beta,-\alpha,-\beta,-1,1) = (\beta, 0 ,-\beta,-1,1) \pmod{\gcd(V,\alpha)}, \\
(\alpha+\beta,-\alpha,-\beta,-1,1) = (\alpha,-\alpha,0,-1,1) \pmod{\gcd(V,\beta)},
\end{gather*}
as required. 

\item For the case $k=2$, primitive, ($5$-tuple $ (1,-2,\alpha,-2\alpha,1+\alpha)$):
\[
(1,-2,\alpha,-2\alpha,1+\alpha) = (1,-2, -1, 2, 0) \pmod{\gcd(V,\alpha+1)}.
\]

\item For the case $k=2$, nonprimitive, with $5$-tuple
\[
\left(0,\frac{V}2,0,\frac{V}2,0\right)+  (-1,-1,\alpha,-\alpha,2) =
  \left(-1,\frac{V}2-1,\alpha,\frac{V}2-\alpha,2\right)
\]
we have
\[
\left(-1,\frac{V}2-1,\alpha,\frac{V}2-\alpha,2\right) = \left(-1,1,-1,1,0\right) \pmod{2}.
\qedhere
\]
\end{itemize}
\end{proof}

\section{Proof of the main theorem, case $k=3$}
\label{sec:k=3}

We now look at the case $k=3$. That is, let $P= \conv(v_0,\dots,v_4)$ be a hollow cyclic $4$-simplex (later we will add the constraint that $P$ is empty) and $\pi:\R^4\to \R^3$ be a projection map sending the vertices of $P$ to a hollow $3$-dimensional configuration $S=\{s_0,\dots,s_4\}$ with the property that $S$ does not project to dimension two. 
There are finitely many possibilities for $S$, by Theorem~\ref{thm:NillZiegler}. Their exhaustive computation was done in \cite{AverkovKrumpelmannWeltge} and can be summarized as follows:

\begin{theorem}[\cite{AverkovKrumpelmannWeltge}]
\label{thm:AKW}
There are twelve maximal $3$-dimensional hollow polytopes that do not project to dimension two. Their  volumes are bounded by $36$ (attained by the tetrahedron $\conv(0,6e_1, 3e_2, 2e_3)$).
\end{theorem}

Observe that $\conv(S)$ is a $3$-polytope with four or five vertices, for which there are combinatorially three cases: it is either a tetrahedron, a pyramid over a quadrilateral, or a triangular bipyramid (a convex union of two tetrahedra with a common facet).
Our proof mixes a (computationally straightforward) enumeration of the subconfigurations of the twelve maximal $3$-polytopes from Theorem~\ref{thm:AKW} with some theoretical observations. What we need from the enumeration is summarized in the following statement. The computations giving it were done by M\'onica Blanco:

\begin{lemma}
\label{lemma:dim3enum}
The subconfigurations of size five and not projecting to dimension two of the twelve polytopes of Theorem~\ref{thm:AKW} are as follows, according to the combinatorics of their convex hulls: 
\begin{enumerate}
\item A certain number of tetrahedra (with an additional boundary point).
\item 24 quadrilateral pyramids, all of them containing some lattice point in the interior of the quadrilateral facet.
\item 29 primitive bipyramids, whose affine dependences are generated by the $5$-tuples in Table~\ref{table:primitive}.
\item 23 nonprimitive bipyramids, whose data are specified in Table~\ref{tab:nonprimitive-bipyr}.
\qed
\end{enumerate}
\end{lemma}

\begin{table}
\hrule
\setlength\tabcolsep{0pt}
\renewcommand{\arraystretch}{1.4}
\small

\[
\begin{array}{c|cc|cc}
I&a\in \frac1I\Z^5  & b\in \Z^5 & 
a\in \frac1I\Z^5&b\in \Z^5 \\ 
\hline
\multirow{5}{*}{2}
&  \frac12( 0,   0,  1, 1,  0 ) &  
    (3,   -1,  -6,    2,    2  ) &
 \frac12(1,  0,   0,   0,  1 ) &  
     ( 4,   -3,    1,   -4,    2  )  

\\ 
 & \frac12( 0,   0,  1,  0,   1 ) &  
    ( 4,   -2,  -6,    3,    1  ) 
& \frac12(1,  0,   0,   0,  1 ) &  
     ( 2,    3,   -1,  -8,    4  ) \\ 
 & \frac12( 0,  1, 1,  0,   0 ) &  
     ( 1,  -6,    2,   6,   -3  ) &

  \frac12(1,  0,  1,  0,   0 ) &  
     (6,  -8,    4,   -3,    1  ) \\ 
&  \frac12( 0,  1,  0,   0,  1  ) &%
       ( 1,   6,   -4,  -6,    3  ) &
  \frac12(1,  0,   0,   0,  1 ) &  
     ( 4,    3,   -1,  -12,   6  ) \\ 
 & \frac12( 0,  1,  0,   0,  1 ) &  
        ( 3,   -1,    4,  -12,   6  ) &
&
\\ 
\hline\hline
\multirow{2}{*}{4}
&  \frac14(2, 1, 1,  0,   0 ) &  
     (3,  -3,    1,   -2,    1  ) &

 \frac14(0,  1,  1,  0,  2 ) &  
     ( 1,    2,   -1,  -4,    2  ) \\
 & \frac14( 0,  0,  1,  2,  1 ) &  
     ( 1,  -4,    1,   4,   -2  ) &
 \frac14( 0, 1, 1,  0,  2 ) &  
     ( 1,   3,   -1,  -6,   3  )\\ 
%
\hline\hline
\multirow{5}{*}{3}
&  \frac13( 0,   0,  2, 1,  0 ) &  
     (-3,    2,    1,    1,   -1  ) 
 & \frac13(1,  0,  2,  0,   0 ) &  
     ( 3,   -3,    1,   -2,    1  )\\ 
 & \frac13( 0,   0,  1, 2,  0 ) &  
     (-3,    1,    2,    2,   -2  ) &
  \frac13( 0,   0,  1, 2,  0 ) &  
     (4,   -2,  -4,    1,    1  ) \\ 
&  \frac13(1,  0,  2,  0,   0 ) &  
     ( 3,  -6,    2,    2,   -1  ) &
  \frac13(1,  0,  2,  0,   0 ) &  
     (4,  -6,    1,    2,   -1  ) \\ 
&  \frac13(1,  0,  2,  0,   0 ) &  
     (4,   -3,    1,  -4,    2  ) &
  \frac13(1,  0,  2,  0,   0 ) &  
     ( 2,   -1,    2,  -6,    3  ) \\ 
&  \frac13( 0,   0,  1, 1, 1 ) &  
     ( 1,  -6,    2,   6,   -3  ) &

\\ 
\hline\hline
\multirow{1}{*}{6}
&  \frac16(1,   0,   0,  4, 1 ) &  
     ( 1,  -3,    1,   2,   -1  ) \\ 
\end{array}
\]

\medskip
\caption{The 23 non-primitive hollow triangular bipyramids of Lemma~\ref{lemma:dim3enum}.
In each of them we give a generator $a\in \frac1I\Z^5$ for the quotient group $\Z^3/\Lambda_S\cong \Z_I$ (written in barycentric coordinates with respect to the vertex set $S$ of the bipyramid) and the primitive affine dependence $b\in \Z^5$ among $S$.
 }
\label{tab:nonprimitive-bipyr}
\hrule\end{table}

The following statement shows that we do not need to care much about tetrahedra and quadrilateral pyramids:

\begin{proposition}
\label{prop:sporadic-k=3}
Let $S$ be a hollow configuration of five points in $\R^3$ such that $\conv(S)$ is one of the tetrahedra or quadrilateral pyramids of Lemma~\ref{lemma:dim3enum}.  Then, any empty $4$-simplex projecting to $S$ has volume bounded by $72$.
\end{proposition}

\begin{remark}
The existence of a global bound in Proposition~\ref{prop:sporadic-k=3} follows from results in \cite{BlancoHaaseHoffmanSantos}. For tetrahedra this is Corollary 4.2 in that paper, and for pyramids over non-hollow polytopes it is the combination of Corollary 4.4 and Lemma 4.1. We include a proof of Proposition~\ref{prop:sporadic-k=3} in order to give the explicit bound of 72.
\end{remark}

\begin{proof}
Let $P=\conv(v_0,\dots,v_4)$ be empty and projecting to $S=\{s_0,\dots,s_4)$ and let $\pi:\R^4\to \R^3$ be the projection map. (We  assume $\pi(v_i)=s_i$). 

Suppose first that $\conv(S)$ is a tetrahedon, with vertices $s_1,s_2,s_3,s_4$. Let $s_0$ be the fifth element of $S$ (which may or may not coincide with one of the vertices). Since $P$ is not empty, $\pi^{-1}(s_0)\cap P$ is a segment having $v_0$ as one end-point and of length at most one. It is easy to show (see \cite[Lemma 3.1]{IglesiasSantos}; in our case $s_0$ is the ``Radon point of $S$'') that
\begin{gather}
\label{eq:vol-tetra}
\Vol(P) = \Vol(\conv(S)) \times \length(\pi^{-1}(s_0)\cap P) \le \Vol(\conv(S)).
\end{gather}
For the tetrahedra of Lemma~\ref{lemma:dim3enum} this gives us a bound of 36, via Theorem~\ref{thm:AKW}.

Suppose now that $\conv(S)$ is a pyramid over a quadrilateral $Q$, with apex at $s_0$.  Let $\ell$ be the lattice distance between the plane spanned by $Q$ and $s_0$, and let $F=\conv(v_1,v_2,v_3,v_4)$ be the facet of $P$ that projects to $Q$. Observe that the lattice distance between the hyperplane spanned by $F$ and $v_0$ divides $\ell$. In particular, 
\[
\Vol(\conv(S)) = \ell \times \Vol(Q),
\qquad{and}\qquad
\Vol(P) \le \ell \times \Vol(F).
\]

Let $x$ be the intersection of the two diagonals of $Q$, which is in this case the Radon point of $S$ as used in \cite[Lemma 3.1]{IglesiasSantos}.
As before, that lemma says
\[
\Vol(F) = \Vol(Q) \times \length(\pi^{-1}(x)\cap F),
\]
so that
\[
\Vol(P) \le \Vol(\conv(S)) \times \length(\pi^{-1}(x)\cap F).
\]

It is no longer true that $\length(\pi^{-1}(x)\cap F) \le 1$, but we can bound this length as follows. Let $y$ be an interior lattice point in $Q$ and let $z$ be the last point in $Q$ along the ray from $x$ through $y$ (if $x=y$, let $z$ be an arbitrary boundary point of $Q$). Then, it  follows from the proof of \cite[Lemma 3.1]{IglesiasSantos} (see also the related result \cite[Lemma 3.5]{IglesiasSantos}) that
\[
\length(\pi^{-1}(x)\cap F) = \frac{|xz|}{|yz|} \length(\pi^{-1}(y)\cap P) \le \frac{|xz|}{|yz|}.
\]
This gives us the desired upper bound on the volume of $P$:
\begin{gather}
\label{eq:vol-quad}
\Vol(P) \le \Vol(\conv(S)) \times \frac{|xz|}{|yz|}.
\end{gather}
For the 24 pyramids of Lemma~\ref{lemma:dim3enum} this formula (taking the best possibility for the interior point $y$ whenever there is a choice) gives  the claimed bound of 72.
\end{proof}

With this we can now prove the case $k=3$ in our main theorem:

\begin{proof}[Proof of Theorem~\ref{thm:main}, case $k=3$]
Let $P=\conv(v_0,\dots,v_4)$ be an empty $4$-simplex projecting to a hollow configuration $S=\{s_0,\dots,s_4\}\subset \R^3$ that does not project to dimension two. By Proposition~\ref{prop:sporadic-k=3}, if $\conv(S)$ is not a triangular bipyramid then $\Vol(P)\le 72$. For each of the $29$ primitive plus 23 nonprimitive triangular bipyramids of Lemma~\ref{lemma:dim3enum}, Corollary \ref{coro:parameters} tells us how to parametrize the $(d+1)$-tuples of empty simplices projecting to them. More precisely:
\begin{itemize}
\item When the bipyramid is primitive, the $5$-tuple is an integer affine dependence $b$ among $S$. A priori there are different possibilities for $b$, since the affine dependence of $S$ is only unique modulo multiplication by a scalar.
But by Proposition~\ref{prop:coprime} we can assume that $b$ does not have a common divisor with $V$. This implies that $b$ equals, modulo $V$, the primitive affine dependence times a factor coprime with $V$. Since multiplying a $5$-tuple by a unit modulo $V$ does not change the cyclic simplex it represents, there is no loss of generality in taking $b$ to be the primitive dependence as we do in Table~\ref{table:primitive}.

\item When the bipyramid is not primitive, the $5$-tuple is of the form $V a +  b$ where $a$ are the barycentric coordinates of a generator of $\Z^3/\Lambda_S$ and $b$ is an integer affine dependence among $S$. Observe that Corollary \ref{coro:parameters} allows us to choose our preferred $a$ (even if $\Z^3/\Lambda_S$ may have several generators) but it does not, a priori, allow us to choose $b$. But, as before, every two valid choices of $b$ are related via a unit modulo $V$. That is, every empty simplex of volume $V$ for one of these bipyramids can be represented as a $5$-tuple of the form
\[
V a+ r b, 
\]
where $a$ and $b$ are the choices in Table~\ref{tab:nonprimitive-bipyr}, and $r\in \Z$ is coprime with $V$.
Multiplying such a $5$-tuple by $r^{-1} \pmod V$ we find that the same simplex is represented by 
\[
V r^{-1}a +  b.
\]
Now, since $I$ divides $V$, $r$ is also a unit modulo $I$, which implies that $r^{-1}a$ is also a generator of $\Z^3/\Lambda_S$. In all our cases $I=\{1,2,3,6\}$, so $\Z^3/\Lambda_S\cong \Z_I$ has only two generators, $\pm a$. Thus, our simplex is represented by  the $5$-tuple $\pm V a +  b$.
\end{itemize}

This finishes the proof, except for the fact that in Table~\ref{tab:nonprimitive-bipyr} we have 23 non-primitive $5$-tuple while in Theorem~\ref{thm:main} (Table~\ref{table:nonprimitive}) only seventeen appear, and except for the restrictions on $V$ displayed in Tables~\ref{table:primitive} and \ref{tab:nonprimitive-bipyr}. These restriction are proved in Propositions~\ref{prop:primitiverestrictions} and \ref{prop:nonprimitiverestrictions} below and imply, in particular, that the six nonprimitive $5$-tuples that have ``$V\in \emptyset$'' as a restriction in Table~\ref{tab:nonprimitive-bipyr} do not produce any empty $4$-simplex.
\end{proof}

To show that the restrictions on $V$ shown in Tables~\ref{table:primitive} and \ref{tab:nonprimitive-bipyr} are necessary and sufficient for the $5$-tuple to produce an empty 4-simplex, we use Propositions~\ref{prop:coprime} and~\ref{prop:emptyfacets}, as we did in the previous section.

\begin{lemma}
\label{lem:volume_k=3_primitive}
For the empty simplices of  the case ``$k=3$, primitive'', all facets are unimodular except for the 12 $5$-tuples
of Table~\ref{table:volume_k=3_primitive}, which can have up to three nonunimodular facets, as indicated.
\end{lemma}

\begin{proof}
By Lemma~\ref{lemma:facetvolumes}, the volume of the $i$th facet of the primitive cyclic simplex of volume $V$ with $5$-tuple  $b$ equals $\gcd(V,b_i)$. On the other hand, by Proposition~\ref{prop:coprime}, no two facets can have volumes with a common factor.  Thus, the vector of facet volumes divides (coordinate-wise) the vector obtained from $b$ by removing the prime factors that divide two or more entries of $b$. These vectors are precisely what we show in the last column in Table~\ref{table:volume_k=3_primitive}.
\end{proof}

\begin{table}
\hrule
\[
\begin{array}{c|c|c}
\text{$5$-tuple} & \text{condition on } &\text{max.~volumes of facets}\\
& \text{$V$ for emptiness} &\\
\hline
(15,1,-2,-5,-9) & V\not\in 3\Z  & (1,1,2,1,1)  \\ 
(9,7,1,-3,-14)  & V\not\in3\Z\cup 7\Z & (1,1,1,1,2) \\ 
(15,7,-3,-5,-14)& V\not\in3\Z\cup 5\Z\cup 7\Z& (1,1,1,1,2)  \\
\cline{1-3}
(10,8,3,-1,-20)  & V\not\in2\Z\cup 5\Z & (1,1,3,1,1)  \\ 
\cline{1-3}
(12,3,-4,-5,-6) &   & (1,1,1,5,1)  \\ 
(9,6,5,-2,-18)  & V\not\in2\Z\cup 3\Z & (1,1,5,1,1)  \\ 
(12,8,5,-1,-24) & & (1,1,5,1,1)  \\
\cline{1-3}
(12,2,-3,-4,-7) & V\not\in 3\Z  & (1,1,1,1,7)  \\ 
(10,7,4,-1,-20) & V\not\in3\Z\cup 7\Z & (1,7,1,1,1)  \\ 
\cline{1-3}
(8,5,3,-1,-15) & V\not\in3\Z\cup 5\Z & (8,1,1,1,1) \\ 
\hline
(9,1,-2,-3,-5)  &  V\not\in 3\Z & (1,1,2,1,5) \\ 
\hline
(7,5,3,-1,-14)  & V\not\in 7\Z  & (1,5,3,1,2) \\ 
\end{array}
\]
\caption{Possible nonunimodular facets in the case ``$k=3$, primitive''. The facet volumes depend on the actual $V$. More precisely, an entry $(v_0,v_1,v_2,v_3,v_4)$ in the last column means that the volume of the $i$th facet equals $\gcd(V,v_i)$.
}
\label{table:volume_k=3_primitive}
\hrule\end{table}

\begin{proposition}
\label{prop:primitiverestrictions}
The conditions on $V$ stated in  Table~\ref{table:primitive} are necessary and suffcient for the $5$-tuples to represent empty simplices.
\end{proposition}

\begin{proof}
Necessity follows from Proposition~\ref{prop:coprime}, since in all cases the restriction can be restated as ``$V$ has no factor in common with two of the entries in $B$''. Sufficiency follows from  Proposition~\ref{prop:emptyfacets} and the description of facet volumes in Table~\ref{table:volume_k=3_primitive}.
Let us look at the first case in detail and leave the rest to the interested reader. Our $5$-tuple is 
$
 (9,1,-2,-3,-5), 
$
and the worst values for the $(V_0,V_1,V_2,V_3,V_4)$ of Proposition~\ref{prop:emptyfacets} are $(1,1,2,1,5)$, as expressed in Table~\ref{table:volume_k=3_primitive}. We say ``worst'' because $V_2$ is only 2 if $V$ is even and $V_4$ is only 5 if $V\in 5\Z$, but if this is not the case then the corresponding $V_i$ equals $1$ and the condition in part (2) of Proposition~\ref{prop:emptyfacets} is void. So, assuming the worst case, what we need to check is that
\[
(9,1,-2,-3,-5) = (-1,1,-2,2,0)  \pmod{5}
\]
and
\[
(9,1,-2,-3,-5) =  (-1,1,0,-1,1)   \pmod{2}
\]
have their non-zero entries forming two pairs of opposite residues modulo the respective $V_i\in \{5,2\}$, which is indeed the case.
\end{proof}

\begin{table}
\hrule
\[
\begin{array}{c|c|c|c|c}

\multirow{2}{*}{$I$}&\pm Va + b, \text{written in}& \multicolumn{2}{c|}{\text{conditions on $\pm k$}}&  \text{max.~vol.} \\
\cline{3-4}
&  \text{terms of $k:=V/I$} &  \text{mod 2} & \text{mod 3}&  \text{of facets}\\
\hline
\multirow{9}{*}{2}
&(3,-1,\pm k-6,\pm k+2,2)&=1&\ne0& (1,1,1,1,{\bf 2})\\
&(\pm k+4,-3,1,-4,\pm k+2)&=1&& (1,3,1,{\bf 2},1)\\
&(4,-2,\pm k-6,3,\pm k+1)&\in \emptyset&& \\
&(\pm k+2,3,-1,-8,\pm k+4)&=1&& (1,3,1,{\bf 2},1) \\
&(1,\pm k-6,\pm k+2,6,-3)&=1&\ne0& (1,1,1,{\bf 2},1) \\
&(\pm k+6,-8,\pm k+4,-3,1)&=1&\ne0& (1,{\bf 2},1,1,1) \\
&(1,\pm k+6,-4,-6,\pm k+3)&\in \emptyset&& \\
&(\pm k+4,3,-1,-12,\pm k+6)&=1&\ne0& (1,1,1,{\bf 2},1) \\
&(3,\pm k-1,4,-12,\pm k+6) &\in \emptyset&& \\
\hline
\multirow{4}{*}{4}
&(\pm 2k+3,\pm k-3,\pm k+1,-2,1)&=0&\ne0& (1,1,1,{\bf 2},1)\\
&(1,\pm k+2,\pm k-1,-4,\pm 2k+2)&\in \emptyset&& \\
&(1,-4,\pm k+1, \pm 2k+4,\pm k-2)&\in \emptyset&&\\
&(1,\pm k+3,\pm k-1,-6,\pm 2k+3)&=0&\ne0& (1,1,1,{\bf 2},1) \\
\hline
\multirow{9}{*}{3}
&(-3,2,\pm 2k+1,\pm k+1,-1)&&=0& ({\bf 3}, 2,1,1,1) \\
&(\pm k+3,-3,\pm 2k+1,-2,1)&&=2& (1,{\bf 3},1, 2,1) \\
&(-3,1,\pm k+2,\pm 2k+2,-2)&=1&=0& ({\bf 3},1,1,1,1) \\
&(4,-2,\pm k-4,\pm 2k+1,1)&=1&\ne 1& (1,1,1,1,1) \\
&(\pm k+3,-6,\pm 2k+2,2,-1)&=1&=1& (1,{\bf 3},1,1,1) \\
&(\pm k+4,-6,\pm 2k+1,2,-1)&=1&=0& (1,{\bf 3},1,1,1) \\
&(\pm k+4,-3,\pm 2k+1,-4,2)&=1&=0& (1,{\bf 3},1,1,1) \\
&(\pm k+2,-1,\pm 2k+2,-6,3)&\in \emptyset&& \\
&(1,-6,\pm k+2,\pm k+6,\pm k-3)&=1&=2& (1,{\bf 3},1,1,1) \\
\hline
\multirow{1}{*}{6}
&(\pm k+1,-3,1,\pm 4k+2,\pm k-1)&=0&=0& (1,{\bf 3},1,{\bf 2},1)\\
\end{array}
\]
\caption{The first column shows the 23 possibilities for the $5$-tuple $\pm Va + b$ for cyclic simplices in the case ``$k=3$, nonprimitive''
(compare with Table~\ref{tab:nonprimitive-bipyr}). 
The second and third columns the restrictions on $k:=V/I$ that make the simplex empty. The last column shows the possible facet volumes. 
As in Table~\ref{table:volume_k=3_primitive}, an entry $(v_0,v_1,v_2,v_3,v_4)$ in the last column means that the volume of the $i$th facet equals $\gcd(V,v_i)$. When $v_i$ divides the index $I$ one automatically has $\gcd(V,v_i)=v_i$. In this case the corresponding  entry $v_i$ is marked in boldface.
}
\label{table:volume_k=3_nonprimitive}
\hrule\end{table}

\begin{proposition}
\label{prop:nonprimitiverestrictions}
Let $a$ and $b$ be one of the 23 possibilities in Table~\ref{tab:nonprimitive-bipyr}. Let $I\in\{2,3,4,6\}$ be its index. Let $k\in \N$ and $V=kI$.
The following are equivalent:
\begin{enumerate}
\item $k$ satisfies the restrictions modulo $2$ and $3$ stated in Table~\ref{table:volume_k=3_nonprimitive}.
\label{itm:npr-1}
\item No factor of $V$ divides two entries of $\pm V a+b$.
\label{itm:npr-2}
\item The simplex of volume $V$ represented by the $5$-tuple $\pm V a +  b$ is empty.
\label{itm:npr-3}
\end{enumerate}
\end{proposition}

\begin{proof}
In the first column of Table~\ref{table:volume_k=3_nonprimitive} we have written the vector $\pm Va +b$, in terms of $k$. 
Observe that we have $\pm Va +b= \pm k a' +b$, where $a':=Ia$ is the integer vector from the first column of Table~\ref{tab:nonprimitive-bipyr}.
From this
the reader can easily check the implication \eqref{itm:npr-2}$\Leftrightarrow$\eqref{itm:npr-1}; if $k$ does not satisfy one of the restrictions, then $2$ or $3$ is a common factor of $V=kI$ and at least  two entries of $\pm a+\fracV  b$.
For the reverse implication, first observe that if a prime $p$ divides $kI$ and some entry of $\pm V a+  b$ then $p\in\{2,3\}$; indeed, if $p$ divides $I\in \{2,3,4,6\}$ then this is obvious and if $p$ divides $k$ then for it to divide an entry of 
$\pm Va +b= \pm k a' +b$ it must divide the corresponding entry of $b$, and these have only $2$ and $3$ as prime factors.
Once this is established, it is clear that the conditions for part \eqref{itm:npr-2} can all be expressed as restrictions on $k$ modulo $2$ and $3$, and direct inspection shows that they are the ones stated in the table. Let us show this in a couple of cases and leave the rest to the reader:

\begin{itemize}

\item  For the forth $5$-tuple of index $3$,
$
a = \frac13(0,  0,  1,  2,   0 ), b= (4,   -2,    -4,  1,    1  ),
$
we have that the $5$-tuple is
\[
\pm Va +b = 
(4,   -2,    \pm k-4,  \pm2k+1,    1).
\]

Since the first two entries are even $V$, hence $k$, must be odd. Modulo three, the third and forth entries of $Va+b$ are multiples of three if $k=1\pmod 3$ and the same holds for $-Va+b$ if $k=-1\pmod 3$. In the table we abbreviate this as $\pm k\ne 1\pmod 3$ meaning that the plus sign is taken for $Va+b$ and the minus sign for $-Va+b$. 

The interpretation of this is that for $k=0\pmod 3$ (that is, $V=0\pmod 9$) this case produces \emph{two} empty simplices, with $5$-tuples $Va+  b$ and $-Va+  b$, while for $k\ne 0 \pmod 3$  it produces only one of the two.
 

\item Consider now the second $5$-tuple of index four, with 
$
a= \frac14(0,  1,  1,  0,  2 )$  and 
$ b= ( 1,    2,   -1,  -4,    2  )$. We have that
\[
\pm Va +b = 
  ( 1, k+2, k-1, -4, 2k+2 ).
\]
It turns out that no matter what the value of $k$ is, this $5$-tuple contains two even entries (the 4th one is always even, the 3rd and 5th are even when $k$ is respectively odd and even). Thus, condition \eqref{itm:npr-2} is never satisfied. In the table we mark this by putting $\pm k\in \emptyset$ as the restriction modulo 2.

This implies, by Proposition \ref{prop:coprime},  that these simplices are not empty, no matter what the value of $V$ is). 
The same happens for the other five $5$-tuples that contain the restriction $\pm k\in \emptyset$.
\end{itemize}

The implication \eqref{itm:npr-3}$\Rightarrow$\eqref{itm:npr-2} is Proposition~\ref{prop:coprime}, so we need to show only \eqref{itm:npr-2}$\Rightarrow$\eqref{itm:npr-3}. 
Part \eqref{itm:npr-2} implies that for each nonunimodular facet, of volume $V_i$, the vector $\pm Va +b$ has a single zero entry, modulo $V_i$. The condition in part (2) of Proposition~\ref{prop:emptyfacets} is then automatic:
modulo $V_i\in \{2,3\}$, every four non-zero integers adding up to zero form two opposite pairs. 
\end{proof}

\begin{corollary}
\label{coro:volume_k=3_primitive}
The facet volumes of the nonprimitie empty $4$-simplices with $k=3$ are as indicated in Table~\ref{table:volume_k=3_nonprimitive}.
\qed
\end{corollary}

\section{Proof of Theorem~\ref{thm:sporadic} (case $k=4$)}
\label{sec:k=4}

As explained in Section 4 of~\cite{IglesiasSantos}, we have enumerated all empty $4$-simplices of volume up to 7600. 
By discarding from the output the simplices in the 1+2+46 infinite families of cases $k=1,2,3$, we have found the 2461 sporadic simplices described in Table~\ref{table:sporadic}. 
That this list is complete, which is the content of Theorem~\ref{thm:sporadic}, follows from the following statement. 
Its proof occupies the rest of this section.

\begin{theorem}
\label{thm:bound}
Let $P$ be a hollow $4$-simplex of width two and which does not project to a hollow $3$-polytope. Then, $\Vol(P)\le 5184$.
\end{theorem}

For simplices of width at least three this statement is~\cite[Theorem 3.6]{IglesiasSantos}. Since the width of $P$ cannot be one (for that would imply $P$ to project to a hollow $1$-polytope) in the rest of the section we assume that $P$ is an empty $4$-simplex of width two. Thus, without loss of generality, we take 
 $P\subset \R^3\times[-1,1]$. 
 
 Let $Q:=P\cap\{x_4=0\}$ be the middle 3-dimensional slice with respect to the last coordinate. If we get a good bound for the volume of $Q$ then we can transfer it to $P$ via the following lemma. A more general version of it appears in \cite{Bernardo}:

\begin{lemma}
\label{lemma:slice}
Let $K\subset \R^d$ be a convex body with supporting hyperplanes $\R^{d-1}\times \{-a\}$ and $\R^{d-1}\times \{b\}$, with $0< a \le b$. Let $K_0:=K\cap \{x_d=0\}$. Then,
\[
\Vol(K) \le a \left(\frac{a+b}a\right)^d \Vol(K_0).
\]
\end{lemma}

\begin{proof}
The proof is based on  applying Schwarz symmetrization (see, e.g.,~\cite[Sect.~9.3]{Gruber}) to our convex body $K$. 

For each $t\in [-a,b]$ let $K_t:=K\cap \{x_d=t\}$, and let $B_t\subset \R^{d-1}$ be the Euclidean ball centered at the origin $O\in \R^{d-1}$ and with the same volume as $K_t$. The \emph{Schwarz symmetrization} of $K$ is defined to be
\[
K^S:= \bigcup_{t\in [-a,b]} B_t \times \{t\}.
\]
Then, $K^S\subset \R^d \times [-a,b]$ is a convex body (as proved by Schwarz), it has the same volume as $K$, and it is symmetric around the line $\{O\} \times \R$. In particular, $K^S$ is contained in a truncated cone $C$ of the form
\[
C= \conv( C_{-a} \times \{-a\} \bigcup C_{b} \times \{b\}),
\]
where $C_{-a}$ and $C_{b}$ are two Euclidean balls with the property that the slice at $t=0$ of $C$ coincides with that of $K^S$. (To prove this, consider a supporting hyperplane of $K^S$ at a boundary point with $t=0$ and rotate it around the line $O\times \R$). 

Let $r$ be the radius of $K_0$, and let $r+\lambda t$ be the radius of $C\cap \{x_d=t\}$. Then,
\[
\Vol(K) \le \Vol(C) = d \int_{-a}^b \kappa_{d-1} (r+\lambda t)^{d-1} \mathrm{d} t = \frac1{\lambda} \left[(r+\lambda b)^d - (r-\lambda a)^d\right] \kappa_{d-1},
\]
where $\kappa_{d-1}$ denotes the normalized volume of the $(d-1)$-dimensional unit ball.  

The slope $\lambda$ must lie between $-r/b$ and $r/a$ (the values for which the truncated cone is actually a cone). 
Within this range the maximum of the right-hand-side is achieved for $\lambda=r/a$, where we have
\[
\Vol(K) = \frac a{r} (r+r b/a)^d  \kappa_{d-1} = a r^{d-1} \kappa_{d-1} \left(\frac{a+b}a\right)^d = a \Vol(K_0) \left(\frac{a+b}a\right)^d.
\]
\end{proof}

\begin{corollary}
\label{coro:PtoQ}
Let $P\subset\R^3\times [-1,1]$ be a polytope of width two and let $Q=P\cap\{x_4=0\}$. Then,
\[
\Vol(P) \le 16 \Vol(Q).
\]
\end{corollary}

To bound the volume of $Q$ we now observe that $Q$ is a hollow $3$-polytope with half-integer vertices; in particular, its width is half-integer. We also know  that $Q$ does not project to a hollow $2$-polytope (otherwise $P$ would project to a hollow 3-polytope). We distinguish three cases:

\begin{enumerate}
\item[(I)] $\text{width}(Q)\ge 5/2$, then by \cite[Thm.~2.1]{IglesiasSantos} we have the following bound.
\[
\Vol(Q) \le \frac{3!\cdot 8\cdot w^3}{(w-1)^3} \le \frac{6\cdot 8\cdot(5/2)^3}{(3/2)^3} =\frac{2000}9=222.22.
\]
(Remark: in fact the bound could be multiplied by $2/3$ using~\cite[Cor. 2.4]{IglesiasSantos}, since $Q$ has at most five vertices). Via Lemma~\ref{lemma:slice} we get that 
\[
\Vol(P) \le 16 \Vol(Q) \le \frac{16000}9 =3555.55.
\]

\item[(II)] If $\text{width}(Q)\le 3/2$, or $\text{width}(Q) = 2$ with respect to a functional whose minimum and maximum are integer, then we assume without loss of generality that $Q\subset \R^2\times[-1,1]\times \{0\}$. In this case we can apply  to the slice $R:=Q\cap\{x_3=0\}$ the same ideas that we applied to $P\cap \{x_4=0\}$, since $R$ is hollow and does not project to a hollow segment. See details below.

\item[(III)] If $Q$ has width two, but with respect only to functionals whose minimum and maximum are half-integer, then 
we can assume $Q\subset \R^2\times[-1/2,3/2]\times \{0\}$. There are two integer slices $R:=Q\cap\{x_3=0\}$ and $R':=Q\cap\{x_3=1\}$. We have two subcases:
\begin{enumerate}
\item[(III.a)] If one of $R$ or $R'$ (say $R$) does not project to a hollow segment, we do the same as in case (II). See details below, in particular Corollary~\ref{coro:II+IIIa}.

\item[(III.b)] If both $R$ and $R'$ project to hollow segments, then they are contained in respective lattice bands of width one. These lattice bands have to be of different direction, since otherwise the projection of $Q$ along that direction is hollow. 
\end{enumerate}

\end{enumerate}

In what follows we give details on Cases (II) and (III), obtaining bounds of $324$ and $192$ for the volume of $Q$ (see Corollary~\ref{coro:II+IIIa}.
and  Lemma~\ref{lem:hollow_polygon_bands}). This finishes the proof of Theorem~\ref{thm:bound}, via Corollary~\ref{coro:PtoQ}. In all cases we will resort to $3$- or $2$-dimensional cases of Lemma~\ref{lemma:slice}.

\begin{lemma}\label{lem:hollow_polygon_bands}
If $Q$ is in case (III.b) then $\Vol(Q)\le 192$, hence $\Vol(P) \le 3072$.
\end{lemma}

\begin{proof}
Let us recall our hypotheses: $Q\subset \R^3$ is a $3$-dimensional hollow polytope with supporting hyperplanes $\{x_3=-1/2\}$ and $\{x_3=3/2\}$, and the slices $R:=Q\cap \{x_3=0\}$ and $R'=Q\cap \{x_3=1\}$ both have width one and project to hollow segments, but with respect to different projection directions.

Applying Lemma~\ref{lemma:slice} to $R\subset Q$ with $a=1/2$ and $b=3/2$ we get that
\[
\Vol(Q) \le \frac12 \left(\frac{2}{1/2}\right)^3 \Vol(R) = 32 \Vol(R).
\]

Now, $R$ has width one with respect to a certain direction, and width at most three with respect to a second one. (For the latter, observe that $R$ is contained in a band of width three along the direction of the band of width one containing $R'$). This implies $\Vol(R)\le 6$, from which we deduce $\Vol(Q)\le 192$ and $\Vol(P) \le 192\cdot 16 = 3072$.
\end{proof}

For cases (II) and (III.a) we need to use that the coordinates of vertices of $Q$ are rational with small denominators:

\begin{lemma}
\label{lemma:polygon_coordinates}
In the conditions of cases (II) or (III), all vertices of $R$ and $R'$ have coordinates in $\frac16\Z^2 \cup \frac18\Z^2$.
\end{lemma}

\begin{proof}
In case (III) the situations in $R$ and $R'$ are symmetric to one another, so for the rest of the proof we only look at $R=Q\cap\{x_3=0\}$.
Since $Q$ is the middle slice of a lattice polytope $P$ of width two, $Q$ is a half-integer $3$-polytope. That is, the vertices of $Q$ have integer or half-integer coordinates. Let now $p$ be a vertex of $R$. Either $p$ is also a vertex of $Q$ (in which case it has half-integer coordinates) or $p$ is the intersection of an edge $uv$ of $Q$ with the plane $x_3=0$. 
Let $\lambda\in (0,1)$ be the coefficient such that
\[
p= \lambda u + (1-\lambda) v.
\]

Assume without loss of generality that $u$ is in $\{x_3<0\}$ and $v$ in $\{x_3>0\}$.
In case (II) we have that $u$ has its third coordinate in $\{-1,-\frac12\}$ and $v$ in $\{\frac12,1\}$. This implies that $\lambda \in \{\frac13, \frac12, \frac23\}$. 
In case (III) we have that $u$ has its third coordinate equal to $-\frac12$ and $v$ in $\{\frac12,1,\frac32\}$, which implies
 $\lambda \in \{ \frac12, \frac23, \frac34\}$. 
Since  $u, v\in \frac12\Z^2$, in all cases we get  $p \in \frac16\Z^2 \cup \frac18\Z^2$.
\end{proof}

\begin{lemma}
\label{lemma:polygon_volume}
Let $R$ be a hollow polygon with vertices in $\bigcup_{i\le k}\frac{1}{i} \Z^2$ for an integer $k\ge 1$ and such that $R$ does not project to a hollow segment. Then, 
\[
\Vol(R) \le \frac{(k+1)^2}k.
\]
\end{lemma}

\begin{proof}
Averkov and Wagner~\cite[Theorem 2.2]{AverkovWagner} have given upper bounds for the maximum area of a hollow polygon in terms of its width $w$, depending on whether $w$ lies in $[0,1]$, $[1,2]$ or $[2,1+2/\sqrt{3}]$. (That $1+2/\sqrt{3}\sim 2.15$ is the maximum possible width was previously shown by Hurkens~\cite{Hurkens}). We prove the statement separately in the three cases:

\begin{itemize}
\item If $w\in (0,1]$ then $R$ is contained in a strip of width one, say $R\subset \R\times [\alpha -1, \alpha]$, with $\alpha\in(0,1)$ ($\alpha$ cannot be an integer, because $R$ does not project to a hollow segment). We can assume without loss of generality that $\alpha\le 1/2$ and then, since $R\cap \{x_2=0\}$ has length at most $1$, Lemma~\ref{lemma:slice} implies
\[
\Vol(R) \le \alpha\left(\frac{1}{\alpha}\right)^2 = \frac{1}{\alpha}\le k,
\]
where the last inequality comes from the fact that $\alpha\in \bigcup_{i\le k}\frac{1}{i} \Z^2$.

\item If $w\in (1,2]$ then the bound from~\cite[Theorem 2.2]{AverkovWagner} is
\begin{align}
\Vol(R)\le 
   \frac{w^2}{w-1}.
    \end{align}
(Observe we have multiplied the formula in~\cite{AverkovWagner} by two, since our volume is normalized to the unimodular triangle and theirs is not).
Since $w>1$ must be in $\bigcup_{i\le k}\frac{1}{i} \Z^2$, we have $w\ge (k+1)/k$.
Since the function $\frac{w^2}{w-1}$ is decreasing for $w\le 2$, we get
\[
\Vol(R)\leq\frac{w^2}{w-1} \leq \frac{(k+1)^2/k^2}{1/k} = \frac{(k+1)^2}{k}.
\]

\item If $w\in [2, 1+2/\sqrt{3}]$ then the bound in~\cite{AverkovWagner} implies $\Vol(R) \le 4$. On the other hand $k\ge 2$ (no hollow lattice polygon has width larger than two), so indeed 
\[
\Vol(R) \le 4 \le \frac{(k+1)^2}{k}.
\]
\end{itemize}
\end{proof}

We can now address cases (II) and (III.a) together:

\begin{corollary}
\label{coro:II+IIIa}
If $Q$ is in one of cases (II) or (III.a) then $\Vol(Q)\le 324$, hence $\Vol(P) \le 5184$.
\end{corollary}

\begin{proof}
By Lemma~\ref{lemma:polygon_volume} we have $\Vol(R) \le 81/8$. Moreover, we can apply Lemma~\ref{lemma:slice} to $Q$ and its slice $R$, with $(a,b)\in \{(1/2,1/2), (1/2, 1), (1/2, 3/2), (1,1)\}$. These four cases give, respectively,
\[
a \left(\frac{a+b}a\right)^{3} \in
\left\{4,  \frac{27}2 , 32  ,  8\right\}.
\]
Hence,
\[
\Vol(Q) \le 32 \frac{81}8 = 324, 
\quad
\Vol(P) \le 16\Vol(Q) \le  5184.
\]
\end{proof}

\section{Facet volumes, $h^*$-vectors and Ehrhart polynomials}
\label{sec:facets}

As a tool to study the restrictions on the parameters $\alpha$, $\beta$ and $V$ in Theorem~\ref{thm:main}, in Sections~\ref{sec:k<=2} and \ref{sec:k=3} we have studied the possible facet volumes of empty $4$-simplices in the infinite families. We here complete that information including a summary of the data for sporadic families, and relate it to  $h^*$-vectors and Ehrhart polynomials.

Recall that the \emph{Ehrhart polynomial} of a lattice $d$-polytope $P$ is a degree $d$-polynomial 
$
E(P,t) = E_d\, t^d + \dots + E_0 \in  \Q[t]
$ 
with the property that
\[
E(P,t) = |tP \cap \Lambda|, \quad \forall t\in \N.
\]
Some well-known facts about it are that $E_d= \Vol(P)/d!$, $E_{d-1}=\Surf(P)/2(d-1)!$, where $\Vol$ and $\Surf$ denote the normalized volume and surface area (the sum of normalized volumes of facets). Also, Ehrhart reciprocity states that
\[
E(P,-t) = (-1)^{d} \, |\text{interior}(tP) \cap \Lambda|, \quad \forall t\in \N.
\]

An alternative way of giving the same information is via the $h^*$-vector (or $\delta$-vector) of $P$, a vector $h^*(P)=(h^*_0,\dots, h^*_d) \in \N^{d+1}$ with the property that
\[
\sum_{n=0}^\infty E(P, n) x^n = \frac{h^*_d x^d + \dots + h^*_0}{(1+x)^{d+1}}.
\]
That is, the $h^*$-vector gives (the vector of coefficients of the numerator of the rational function of) the generating function of the sequence $(E(P,n))_{n\in \N}$. See \cite{BeckRobins,Ehrhart,Stanley} for more information on Ehrhart polynomials and $h^*$-vectors, and~\cite{Scott, Stapledon,EhrhartSpanning,BallettiHigashitani, LS19, HNO18} for results on classifying them.

For empty $4$-simplices, the $h^*$-vector admits the following simple expression in terms of volume and surface area:

\begin{proposition}
\label{prop:hstar}
Let $P$ be an empty $4$-simplex of volume $V$ and facet volumes $(V_0,V_1,V_2,V_3,V_4)$. Let $S=V_0+\dots+V_4$ be the surface area of $P$. Then, the $h^*$-vector of $P$  is
\[
h^*_0=1, \quad
h^*_1=0, \quad
h^*_2= \frac{V+S}2-3, \quad
h^*_3=\frac{V-S}2+2, \quad
h^*_4=0.
\]
\end{proposition}

\begin{proof}
From the two coefficients $E_4=V/24$ and  $E_3=S/12$ and the three values $E(P,-1)=0$, $E(P,0)=1$, $E(P,1)=5$ we can recover the whole Ehrhart polynomial, which turns out to be
\[
E_P(n) = \frac{V}{24}n^4 + \frac{S}{12}n^3+\left(\frac32-\frac{V}{24}\right)n^2 + \left(\frac52-\frac{S}{12}\right)n +1.
\]
From this, routine computations give the $h^*$-vector.
\end{proof}

\begin{remark}
The values $h^*_0=1$, $h^*_1=0$, and $h^*_d=0$ hold for empty simplices in arbitrary dimensions, by the following general formulas for arbitrary lattice polytopes~\cite[Section 3.4]{BeckRobins}:
\[
		 h^{*}_0=1, \qquad
		 h^{*}_1=|P\cap \Z^d|-(d+1), \qquad
		 h^{*}_d=|\text{interior}(P)\cap \Z^d|.
\]
Another general formula is  $\sum_{i=0}^{d}h^*_i = \Vol(P)$~\cite[Cor.~3.21]{BeckRobins}, which in the case of empty simplices
directly gives
\[
h^{*}_2+\dots+h^{*}_{d-1}=V-1.
\]
Observe also that Proposition~\ref{prop:hstar} agrees with the \emph{Hibi inequality} $h^*_2\geq h^*_3$~\cite{Hibi}.
\end{remark}

Proposition~\ref{prop:hstar} implies that the Ehrhart polynomial and $h^*$-vector of an empty $4$-simplex is determined by $h^*_2$ and $h^*_3$ or, equivalently, by 
\begin{gather}
\label{eq:h_to_volumes}
h^{*}_2+h^{*}_3 = V-1 \quad
\text{and }\quad
h^{*}_2-h^{*}_3 = S-5.
\end{gather}
These two quantities are nonnegative and measure how far is $P$ from being unimodular or from having unimodular facets. We call them the \emph{volume excess} and the \emph{surface area excess} of $P$.

\begin{figure}
	\centerline{\includegraphics[scale=0.65]{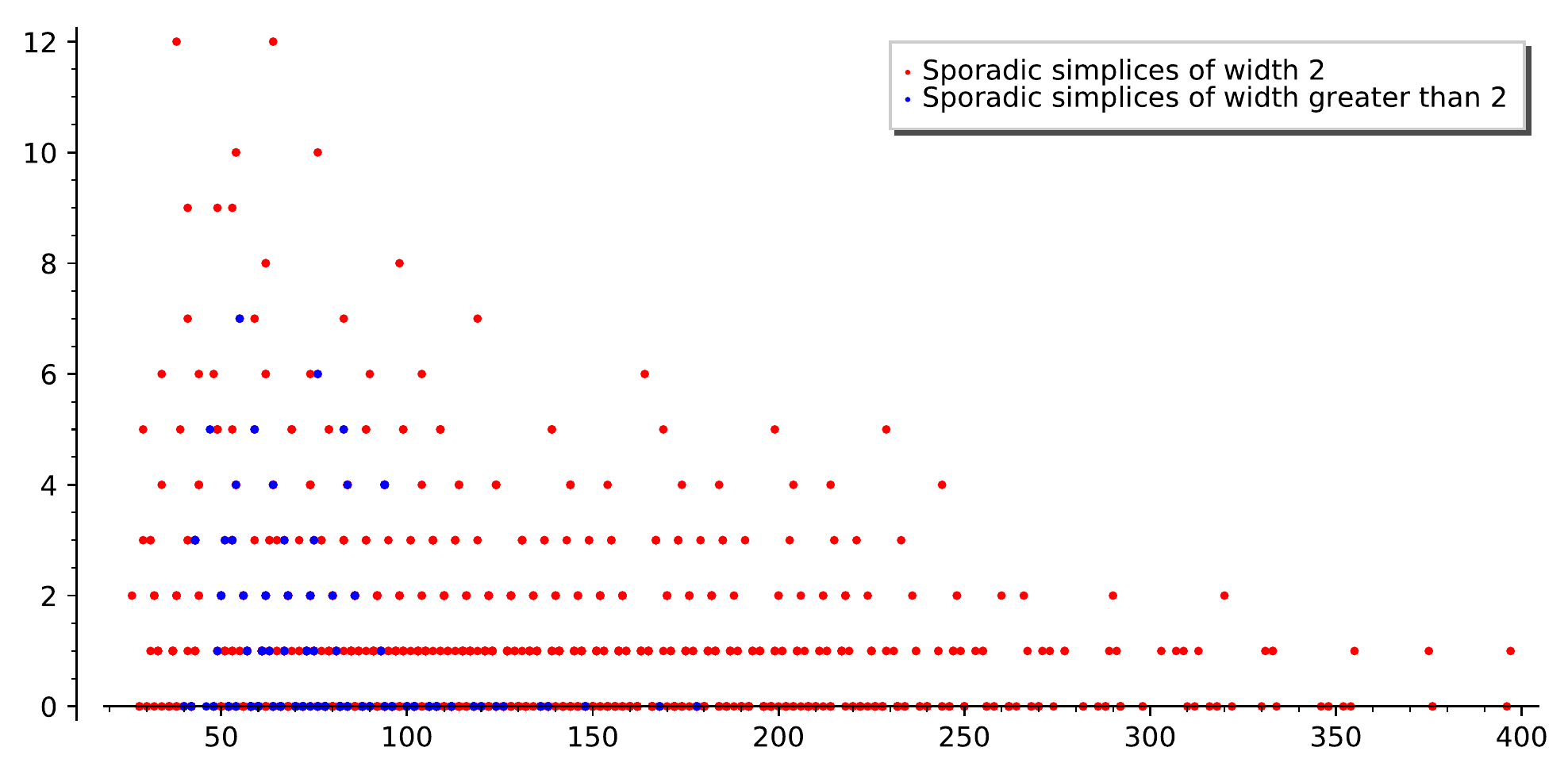}}
	\caption{Possible values of $V-1=h^*_2+h^*_3$ (horizontal axis) and $S-5=h^*_2-h^*_3$ (vertical axis) for the 2461 sporadic empty $4$-simplices}
	\label{fig:sporadic}
\end{figure}

In Figure~\ref{fig:sporadic} we show the statistics of volume and surface area excess for the 2461 sporadic simplices. 
As seen in the figure, the maximum value of the latter is 12. It is achieved exactly twice, for the simplices of volumes 39 and 65 defined by the $5$-tuples $(5,8,13,14,38)$ and $(3,14,23,26,64)$ respectively. They both have width two and a single nonunimodular facet, of volume $13$ in both. 
\footnote{Here and in the rest of the section we mention width of the different examples since this was a crucial invariant for the the bound in Section~\ref{sec:k=4} and in \cite{HZ00,IglesiasSantos}. Observe that $k=1$ is equivalent to width one, $k\in \{2,3\}$ implies width two, and the sporadic simplices can have width between two and four}

\begin{table}[htp]
\hrule
	\begin{tabular}{c|c|c}
		case 
		&  \parbox{3cm}{\centering possible \\ $h^*_2-h^*_3$}
		& \parbox{3cm}{\centering possible $\#$ of non-\\unimodular facets\medskip}\\
		\hline
		$k=1$ & unbounded &$0,1,2,3$ \\
		\hline
		$k=2$, primitive & unbounded& $0,1$\\
		$k=2$, index $2$  & $1$&  $1$\\
		\hline
		$k=3$, primitive & $0,\dots, 7$
		&  $3$ \\
		$k=3$, index $2$& $1, 3$& $2$ \\
		$k=3$, index $3$& $0, 2, 3$& $0,1,2$\\
		$k=3$, index $4$& $1$& $1$\\
		$k=3$, index $6$& $3$& $2$\\
		\hline
		sporadic  & $0,\dots, 12$, except $11$
		&$0,1,2,3$\\
	\end{tabular}
	\vspace{0.2cm}
	\caption{Possible number and volume excess  of nonunimodular facets for empty $4$-simplices depending on the case}
	\label{table:h_and_facets}
\hrule\end{table}
These computations, together with Lemmas~\ref{lemma:volumes1and2}, \ref{lem:volume_k=3_primitive} and Corollary~\ref{coro:volume_k=3_primitive}, give the possibilities for surface area excess and number of nonunimodular facets summarized in  Table~\ref{table:h_and_facets}.

It was stated in the MSc Thesis~\cite{Wessels} (in German) that all empty $4$-simplices have at least two unimodular facets,
but we are not sure that the proof contained there is complete. Our classification provides an alternative proof:

\begin{corollary}
\label{coro:2unimodular}
Every empty $4$-simplex has \emph{at least} two unimodular facets. The ones that have \emph{only} two unimodular facets are:
\begin{itemize}
 \item The simplices with $k=1$ (equivalently, of width $1$) when their $5$-tuple $(\alpha+\beta,-\alpha,-\beta,-1,1)$ has the property that $V$ has prime factors in common with the three of $\alpha$, $\beta$ and $\alpha + \beta$ (such factors are automatically distinct, since $\gcd(\alpha, \beta, V) =1$).
 \item The simplices with $k=3$ (hence, of width two) in the primitive family with $5$-tuple $(7,5,3,-1,-14)$, whenever $V$ is a multiple of $30$.
 \item The following $3$ sporadic empty  $4$-simplices of width two:
\rm
\[
	\begin{array}{c|c|c|c}
		\text{$5$-tuple} & \text{ volume} & \text{facet volumes} & \text{$h^*$-vector} \\
		\hline
		(4,7,15,17,41) & 42 & (2,7,3,1,1) & (1,0,25,16,0)\\
		\hline
		(2,13,21,25,59) & 60  & (2,1,3,5,1) & (1,0,33,26,0)\\
		\hline
		(2,13,25,81,119) & 120 & (2,1,5,3,1) & (1,0,63,56,0)\\
	\end{array}
\]
 \end{itemize}
\end{corollary}
In higher dimension it is no longer true that all empty simplices  have some facet unimodular: there is an empty $5$-simplex of volume $54$ whose facet volumes are $(6,6,9,54,54,54)$~\cite[p.~21]{Wessels}; 
see also \cite[Remark 1]{BBBK11} for a $3$-parameter infinite family of noncyclic empty $5$-simplices projecting to $2\Delta_2$ and with all facets of the same, arbitrarily large, volume.

\section{Towards a classification of hollow $4$-simplices}
\label{sec:hollow2}

Some of the ingredients in the proof of our classification work not only for empty $4$-simplices, but for all hollow ones.  Putting them together we have the following not-so-explicit classification of hollow $4$-simplices.

\begin{theorem}[Classification of hollow $4$-simplices]
\label{thm:mainhollow}
Let $P$ be a hollow $4$-simplex of  volume $V\in \N$ and let $k\in \{1,2,3,4\}$ be the minimum dimension of a hollow polytope that $P$ projects to. Then $P$ belongs to one of the following fine families:

\begin{itemize}

\item[\framebox{$k=1$:}] 
Two fine families projecting to the multisets $\{0,0,0,0,1\}$ and $\{0,0,0,1,1\}$. The cyclic members of these families are parametrized by 
 $5$-tuples of the form $ (\alpha+\beta+\gamma,-\alpha,-\beta,-\gamma,0)$ and  $ (\alpha+\beta,-\alpha,-\beta,-\gamma,\gamma)$, respectively, where $\alpha, \beta, \gamma \in \Z_V$ are arbitrary.

\item[\framebox{$k=2$:}] 
Six fine families projecting to the two multisubsets of $2\Delta_2 \cap \Z^2$ displayed in Figure~\ref{fig:2Delta-cases} or to the following four additional ones:
\medskip
\\
{
\null\hskip -.8cm
\includegraphics[scale=0.7]{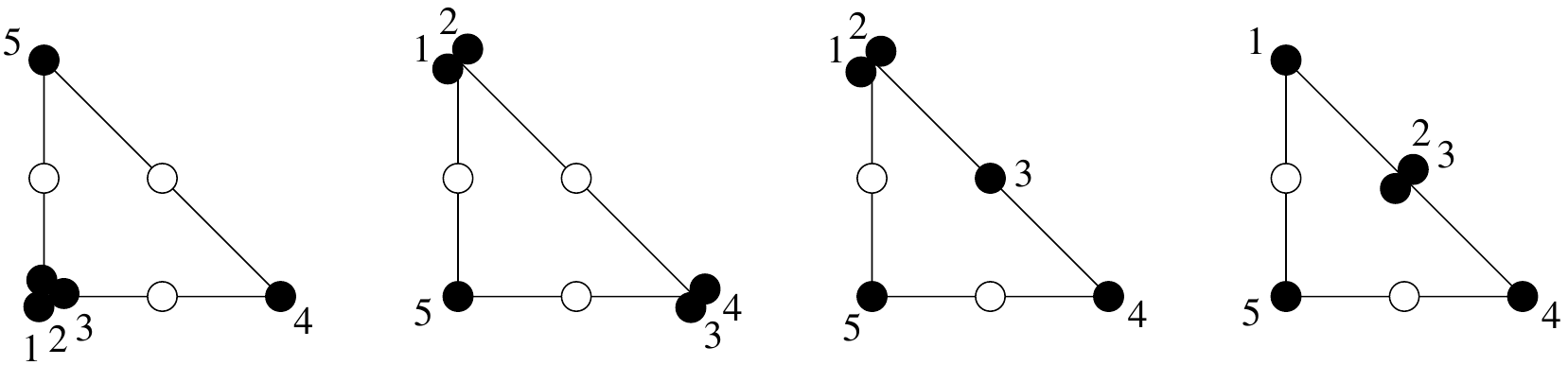}
}
\medskip\\

Cyclic members of the families can be parametrized, respectively, by the following $5$-tuples, where $\alpha, \beta \in \Z_V$ are arbitrary:
\begin{align*}
&\  \,(\beta,-2\beta,\alpha,-2\alpha,\beta+\alpha), \\
\frac{V}2\left(0, 1 ,0, 1, 0\right)\ +&\  (\beta,\beta,\alpha,-\alpha, -2\beta),\\
\frac{V}2\left(0, 0 ,0, 1, 1\right)\ +&\  (\alpha+\beta,-\alpha,-\beta,0,0),\\
\frac{V}2\left(0, 0 ,0, 1, 1\right)\ +&\  (\alpha, -\alpha, \beta,-\beta,0),\\
\frac{V}2\left(0, 0 ,0, 1, 1\right)\ +&\  (\alpha+\beta,-\alpha,-2\beta,\beta,0),\\
\frac{V}2\left(0, 0 ,0, 1, 1\right)\ +&\  (\beta, \alpha-2\beta,-\alpha,\beta,0).
\end{align*}
\item[\framebox{$k=3$:}] $P$ belongs to a finite set of fine families, one corresponding to each tetrahedron, square pyramid, or triangular bipyramid of Lemma~\ref{lemma:dim3enum}. 

\item [\framebox{$k=4$:}] There are finitely many possibilities for $P$, by Theorem~\ref{thm:NillZiegler}. Their volumes are bounded by $5184$.
\end{itemize}
\end{theorem}

\begin{proof}
For $k=1,2$ it is easy to show that the claimed cases exhaust all possibilities for $S$. Let us see this for $k=2$. Once we know that three of the elements of $S$ are the vertices of $2\Delta_2$ the six possibilities come from the fact that the other two can either be also vertices (either the same or two different ones), they can both be midpoints of edges (either the same or two different ones), or they can be a vertex and a midpoint (either opposite or adjacent). The expression for the $5$-tuples follows from Proposition~\ref{prop:tuple} and the easy computation of the spaces of affine dependences of the $2+6$ cases of $S$.
(In all nonprimitive cases the index is two, so there is no choice for the generator $q$ of $\pi(\Lambda)/\Lambda_S$).

For $k=3$ our statement follows from Lemma~\ref{lemma:dim3enum} and the definition of fine family.

For $k=4$ just observe that Theorem~\ref{thm:bound} applies to all hollow $4$-simplices, not only empty or cyclic ones.
\end{proof}

The two missing ingredients to turn Theorem~\ref{thm:mainhollow} into a more explicit description are: 
\begin{itemize}
\item An analysis of what finite non-cyclic groups can arise as $G_P=\Lambda/\Lambda_P$. Since they are (isomorphic to) quotients of $\Z^4$, they can be written as $\Z_{n_1}\oplus\Z_{n_2}\oplus\Z_{n_3}\oplus\Z_{n_4}$ with $n_i$ dividing $n_{i+1}$, $i=1,2,3$. This implies each simplex to be representable by (at most) four $5$-tuples like the ones used in the cyclic case, but we would expect a simpler description to be possible.

\item An enumeration of all hollow $4$-simplices of volume up to $5184$, pruning those that belong to the infinite families with $k\le 3$.
\end{itemize}

\bibliographystyle{amsalpha}

\end{document}